\documentclass[10pt, reqno]{amsart}

\usepackage{amscd,amsmath}
\usepackage{mathrsfs}

\usepackage{amsfonts}
\usepackage{amssymb}
\usepackage{enumerate}
\usepackage{setspace}
\usepackage{color}
\usepackage{esint}
\usepackage{rotating,dsfont,stackengine}
\usepackage{mathabx}
\usepackage{tabu}
\usepackage{tikz} 
\usepackage{setspace}
\usepackage[margin=1in]{geometry}

\usepackage{extpfeil}
\usepackage{ifthen,latexsym,float,colortbl}

\usepackage{scalerel}

\usepackage{caption}
\usepackage{sidecap}

\usepackage[english]{babel}
\usepackage{graphicx}

\definecolor{luh-dark-blue}{rgb}{0.0, 0.313, 0.608}

\usepackage[colorlinks=true, citecolor=blue]{hyperref}

\newtheorem{theorem}{Theorem}[section]
\newtheorem{lemma}{Lemma}[section]

\newtheorem{definition}{Definition}[section]

\newtheorem{remark}{Remark}[section]
\newtheorem{proposition}{Proposition}[section]

\newcommand{\eqn}{\begin{eqnarray}}
\newcommand{\een}{\end{eqnarray}}

\numberwithin{equation}{section}

\DeclareMathOperator{\dv}{div}

\newcommand{\nc}{\newcommand}

\newcommand{\R}{\mathbb R}

\newcommand{\N}{\mathbb N}

\nc{\mc}{\mathcal C}

\newcommand{\bq}{\begin{equation}}
\newcommand{\eq}{\end{equation}}

\newcommand{\lt}{\left}
\newcommand{\rt}{\right}

\newcommand{\pa}{\partial}

\makeatletter
\def\moverlay{\mathpalette\mov@rlay}
\def\mov@rlay#1#2{\leavevmode\vtop{%
   \baselineskip\z@skip \lineskiplimit-\maxdimen
   \ialign{\hfil$\m@th#1##$\hfil\cr#2\crcr}}}
\newcommand{\charfusion}[3][\mathord]{
    #1{\ifx#1\mathop\vphantom{#2}\fi
        \mathpalette\mov@rlay{#2\cr#3}
      }
    \ifx#1\mathop\expandafter\displaylimits\fi}
\makeatother

 \makeatletter
 \newcommand{\norm}{\@ifstar{\@normb}{\@normi}}
 \newcommand{\@normb}[2]{\left\Vert{#1}\right\Vert_{#2}}
 \newcommand{\@normi}[2]{\Vert{#1}\Vert_{#2}}
 \makeatother

\stackMath

\newcommand\wh[1]{%
\savestack{\tmpbox}{\stretchto{%
  \scaleto{%
    \scalerel*[\widthof{\ensuremath{#1}}]{\kern-.6pt\bigwedge\kern-.6pt}%
    {\rule[-\textheight/2]{1ex}{\textheight}}
  }{\textheight}%
}{0.5ex}}%
\stackon[1pt]{#1}{\tmpbox}%
}

\begin{document}

\title[Eventual regularity and asymptotic behavior of weak solutions for Hall-MHD]{Eventual regularity and asymptotic behavior of Leray-Hopf weak solutions for the Hall-MHD system}

\author[Jung]{Jinwook Jung}
\address[Jinwook Jung]{\newline Department of Mathematics and Research Institute for Natural Sciences \newline
Hanyang University, 222 Wangsimni-ro, Seongdong-gu, Seoul 04763, Republic of Korea}
\email{jinwookjung@hanyang.ac.kr}

\author[Shin]{Jaeyong Shin}
\address[Jaeyong Shin]{\newline Department of Mathematics \newline
Yonsei University, Seoul, Republic of Korea}
\email{sinjaey@yonsei.ac.kr}

\date{\today}

\begin{abstract}
In this paper, we study the incompressible, viscous and resistive Hall-magnetohydrodynamic (Hall-MHD) system. We first prove that every two-dimensional Leray-Hopf weak solution becomes smooth after a finite time. In three dimensions, where eventual smoothness for arbitrary Leray-Hopf weak solutions is not known, we construct Leray-Hopf weak solutions for which the magneto-vorticity field $B+\nabla\times u$ eventually gains additional regularity. Finally, under suitable low-frequency pseudomeasure assumptions on initial data, we establish decoupled algebraic decay rates for the velocity and magnetic fields by combining a generalized Fourier splitting method with the eventual smoothness in two dimensions and strong regularity in three dimensions.
\end{abstract}

\maketitle

\vspace{-5ex}

\section{Introduction}\label{sec:intro}
In this paper, we study the incompressible, viscous, and resistive Hall-magnetohydrodynamics (Hall-MHD) system. Specifically, let $u$ be the velocity field of an electrically conducting fluid, $B$ be the magnetic field, and $p$ be the pressure. Then our system reads as follows:
\begin{subequations}\label{HMHD}
\begin{align}
& \pa_t u  +(u\cdot \nabla) u +\nabla p-\nu\Delta u=(\nabla \times B)\times B, \label{HMHD u}\\
&\pa_t B  -\nabla \times (u\times B) -\eta\Delta B +\nabla \times ((\nabla \times B)\times B)=0, \label{HMHD B} \\
& \dv u=0, \quad \dv  B=0,
\end{align}
\end{subequations}
where $u: [0,\infty)\times \R^d\rightarrow\R^3$ is the velocity field of an electrically conducting fluid, $B: [0,\infty)\times \R^d\rightarrow\R^3$ is the magnetic field, $p:[0,\infty)\times\R^d\rightarrow \R$ is the pressure, and $\nu,\eta>0$ are the coefficients of the fluid viscosity and the magnetic resistivity, respectively. We consider the Cauchy problem of \eqref{HMHD} subject to initial data:
\[
(u(0,x), B(0,x)) = (u_0(x), B_0(x)), \quad x \in \R^d.
\] 

The highest order nonlinear term $\nabla\times ((\nabla\times B)\times B)$ in \eqref{HMHD} is called the Hall current term, which was introduced by Lighthill \cite{Lighthill} to describe the effect of magnetic fields on electric currents. This term is widely recognized as essential for describing many phenomena in plasma physics \cite{Balbus, Forbes, Homann, Mininni, Shalybkov, Shay, Wardle}. Due to its physical significance and mathematical challenges, the Hall-MHD system has been the subject of extensive research in recent decades (see e.g. \cite{Acheritogaray, BJS, BK, BKS, BKS2, Chae Degond Liu, Chae Lee, Chae Schonbek, DL19, DT22, Tan}
).

As in the incompressible Navier-Stokes and classical MHD equations (\eqref{HMHD} without the Hall current term), the study of weak solutions and their long-time behavior is of fundamental importance. In this paper, we focus on \textit{Leray-Hopf weak solutions} of \eqref{HMHD}. 

\begin{definition}\upshape 
Let $(u_0,B_0)\in L^2_\sigma(\R^d):=\{ f \in L^2(\R^d) \ : \ \nabla \cdot f = 0  \}$. We call $(u,B)$ a global Leray-Hopf weak solution of \eqref{HMHD} if 
	\begin{enumerate}
	\item $u(t)$ and $B(t)$ are weakly divergence-free for almost every $t>0$, and $(u,B)$ solves the first two equations of \eqref{HMHD} in the sense of distributions with initial data $(u_0,B_0)$.
	\item $(u,B)$ belongs to $L^\infty([0,\infty);L^2(\R^d))\cap L^2(0,\infty;\dot{H}^1(\R^d))$;
	\item the (strong) energy inequality holds
\begin{equation} \label{energy ineq}
\|u(t)\|_{L^2}^2+\|B(t)\|_{L^2}^2+2\nu\int^t_s\|\nabla u(\tau)\|_{L^2}^2\,d\tau+2\eta\int^t_s\|\nabla B(\tau)\|_{L^2}^2\,d\tau\leq \|u(s)\|_{L^2}^2+\|B(s)\|_{L^2}^2,
\end{equation}
for almost all $s>0$ with $s=0$ and all $t\geq s$.
	\end{enumerate}
\end{definition}

The precise weak formulation is recalled in Appendix \ref{sec:app.C}. We also note that especially when $s=0$, \eqref{energy ineq} reduces to 
\begin{equation}\label{weak_en_ineq}
\|u(t)\|_{L^2}^2+\|B(t)\|_{L^2}^2+2\nu\int^t_0\|\nabla u(\tau)\|_{L^2}^2\,d\tau+2\eta\int^t_0\|\nabla B(\tau)\|_{L^2}^2\,d\tau\leq \|u_0\|_{L^2}^2+\|B_0\|_{L^2}^2=:\mathcal{E}_0.
\end{equation}

The notion of Leray-Hopf weak solutions for \eqref{HMHD} is directly motivated by the celebrated construction for the incompressible Navier-Stokes equations by Leray \cite{Leray} and Hopf \cite{Hopf}. As in the incompressible Navier-Stokes equations, the existence of global weak solutions of \eqref{HMHD} is guaranteed \cite{Chae Degond Liu} for $d=3$, while the questions of uniqueness and regularity remain largely open. These issues become more intricate in the presence of the Hall current term, which introduces higher-order nonlinear interactions. In three dimensions, however, the Hall current term obstructs a direct extension of the classical eventual regularity argument known for the Navier-Stokes and classical MHD equations.

Then we ask the same questions for $d=2$. For the Navier-Stokes equations, one sees that any Leray-Hopf weak solution is unique \cite{La59} and becomes instantly smooth, while the uniqueness and eventual smoothness remains open for the Hall-MHD system. Note that the Hall-MHD system for $d=2$ or its 2$\frac12$D representation is written as follows (see \cite{Chae Lee}):

\begin{subequations}\label{2.5HMHD}
\begin{align}
& \pa_t u  +u\cdot \widetilde{\nabla} u +\nabla p-\nu\Delta u=(\widetilde{\nabla} \times B)\times B,\label{2.5Hall1} \\
&\pa_t B  -\widetilde{\nabla} \times (u\times B) -\eta\Delta B +\widetilde{\nabla} \times ((\widetilde{\nabla} \times B)\times B)=0, \label{2.5Hall2} \\
& \widetilde{\nabla}\cdot u=0, \quad \widetilde{\nabla}\cdot B=0,
\end{align}
\end{subequations}

where  $u = (u_1, u_2, u_3): \R^2 \to \R^3$, $B=(B_1, B_2, B_3): \R^2 \to \R^3$ and $\widetilde\nabla:= (\pa_1, \pa_2, 0)$, respectively. For this model, the local existence of strong solutions with large data, and the global existence with small initial data are shown \cite{BKS, DT22, Tan}.

The main objective of this paper is to investigate the eventual regularity and long-time asymptotic behavior of such global Leray-Hopf weak solutions. Our main results are threefold and can be summarized as follows:  
\begin{enumerate}
\item
Eventual smoothness in two dimensions,  
\item
Eventual regularity of $(B+\nabla\times u)$ in three dimensions,
\item
Decay rates via a generalized Fourier splitting method.
\end{enumerate}

\vspace{2ex}
\subsection{Eventual smoothness in two dimensions}
Our first result concerns the eventual smoothness of Leray-Hopf weak solutions of \eqref{HMHD} for $d=2$ which refers to \eqref{2.5HMHD}. 

\begin{theorem}\label{thm:2D-eventual-regularity}\upshape
Let $(u,B)$ be a global Leray-Hopf weak solution of \eqref{HMHD} for $d=2$ corresponding to $(u_0,B_0)\in L^2_\sigma(\R^2)$. Then, there exists a time $T_0=T_0(\|u_0\|_{L^2},\|B_0\|_{L^2},\nu,\eta)>0$ such that the solution $(u,B)$ becomes smooth after $t=T_0$, \textit{i.e.}, $(u,B)\in C([T_0,\infty);H^1(\R^2))\cap C((T_0,\infty);C^{\infty}(\mathbb{R}^2))$.
\end{theorem}
The mechanism behind this eventual smoothness lies in the fact that, in two dimensions, the energy dissipation in \eqref{energy ineq} controls the scaling-invariant norm of the magnetic field. In fact, \eqref{HMHD} does not admit a natural scaling invariance. Instead, one can consider the scaling-invariant norm of the electron MHD equations
\begin{equation}\label{EMHD}
\pa_t B-\eta\Delta B+\nabla\times((\nabla\times B)\times B)=0, \quad \dv B=0.
\end{equation}
\eqref{EMHD} contains the diffusion term and the Hall current term, the main part of \eqref{HMHD B}, and its scaling-invariant norm is $\dot{H}^{\frac{d}{2}}$ for $d=2, 3$. Hence, for $d=2$, the scaling-invariant norm coincides with the order of energy dissipation in \eqref{energy ineq}, which ensures that the $\|\nabla B\|_{L^2}$ norm of Leray-Hopf weak solutions eventually becomes small. In contrast, for $d=3$ since the scaling-invariant norm is stronger than $\dot{H}^1$, eventual smoothness would not be deduced from the energy inequality.

\subsection{Eventual regularity of magneto-vorticity fields in three dimensions}
Secondly, we consider Leray-Hopf weak solutions to \eqref{HMHD} for $d=3$. Even though Leray-Hopf weak solutions to \eqref{HMHD} might not become eventually smooth, we can show there exist solutions for which $(B+\nabla\times u)$ eventually possesses additional regularity. From now on, we write $\omega:= \nabla \times u$ for simplicity.

\begin{theorem}\label{thm:3D_weak_sol}\upshape
Let $d=3$ and $(u_0,B_0)\in L^2_\sigma(\R^3)$. Then, there exists a global Leray-Hopf weak solution $(u,B)$ of \eqref{HMHD}, which satisfies the followings:
\begin{enumerate}
\item There exists $T_0(\mathcal{E}_0,\nu,\eta)>0$ such that $B+\omega\in L^\infty([T_0,\infty);L^2(\R^3))$ and $\nabla(B+\omega)\in L^2(T_0,\infty;L^2(\R^3))$ and $\|(B+\omega)(t)\|_{L^2}\rightarrow 0$ as $t\rightarrow \infty$.
\item Additionally if $\nu=\eta$, then $B+\omega\in L^\infty([T_0,\infty);H^2(\R^3))$ and $\nabla(B+\omega)\in L^2(T_0,\infty;H^2(\R^3))$ and $\|\nabla^k (B+\omega)(t)\|_{L^2}^2\le C(1+t-T_0)^{-k}$ for $t\ge T_0$ and $k=1,2$.
\end{enumerate}
\end{theorem}

This result is motivated by the fact that the regularity of $B+\omega$ can be controlled by the energy dissipation. Taking the curl operator to \eqref{HMHD u} and adding \eqref{HMHD B}, we derive the equation of $B+\omega$:
\begin{equation}\label{eq of B+omega}
\pa_t(B+\omega)+u\cdot\nabla(B+\omega)-(B+\omega)\cdot\nabla u=\nu\Delta\omega+\eta\Delta B.
\end{equation}
The disappearance of the Hall current term in \eqref{eq of B+omega} suggests that $B+\omega$ can exhibit better regularity properties, especially when $\nu=\eta$. It was shown in \cite{BKS} that $B+\omega$ is $L^2$-regular (\textit{i.e.}, $B+\omega\in L^\infty_tL^2_x\cap L^2_t\dot{H}^1_x$) or $H^2$-regular when $\nu=\eta$ as long as the fluid velocity $u$ satisfies Serrin type condition: $u\in L^2_t\text{BMO}_x$, or $u\in L^q_tL^p_x$ for $3/p+2/q\le 1$. Moreover, as a corollary of the result, we obtain that if 
\begin{equation}\label{condition of B+omega}
\|(B+\omega)(t_0)\|_{L^2}^2+(\|u(t_0)\|_{L^2}^2+\|B(t_0)\|_{L^2}^2)\frac{(\nu-\eta)^2}{\nu\eta}\ll 1 \quad \mbox{at some } \ t_0>0,
\end{equation}
then $B+\omega$ is $L^2$-regular for all $t\ge t_0$. Hence, by an argument similar to that for Theorem \ref{thm:2D-eventual-regularity}, the energy inequality \eqref{energy ineq} and the long-time decay of $\|u(t)\|_{L^2}^2+\|B(t)\|_{L^2}^2$ (Proposition \ref{prop:non-uniform}) imply that Leray-Hopf weak solutions constructed in Theorem \ref{thm:3D_weak_sol} eventually satisfy the condition \eqref{condition of B+omega}. However, we do not guarantee that Theorem \ref{thm:3D_weak_sol} holds for all Leray-Hopf weak solutions, due to the lack of uniqueness in the continuation of weak solutions. Nevertheless, we can conclude that the blow-up scenario of $B+\omega$ for strong solutions introduced in \cite{BKS} cannot occur at infinite time, but only possibly in finite time.

\subsection{Decoupled decay rates}
Finally, we study algebraic decay rates for solutions of \eqref{HMHD} by using a variant of Schonbek's Fourier splitting method. In \cite{BJS}, the authors of the current paper and their collaborator generalized the Fourier splitting method \cite{Schonbek} by replacing the usual $L^1$ assumption with a pseudomeasure condition.\footnote{Define $\mathcal{Y}^\sigma(\R^d):=\{f\in\mathcal{S}'(\R^d): \sup_{\xi\in\R^d}|\xi|^\sigma|\widehat f(\xi)|<\infty\}$. Then $L^1(\R^d)\hookrightarrow\mathcal{Y}^0(\R^d)$.} Since the Fourier splitting argument uses only low-frequency information on the initial data, it is sufficient for our purposes to impose the
following low-frequency pseudomeasure condition: 
\[
\mathcal{Y}^\sigma_l(\R^d) := \left\{ f\in\mathcal{S}'(\R^d) \;:\; \sup_{|\xi|\le 1}|\xi|^\sigma|\widehat f(\xi)|<\infty
\right\}.
\]

With this notation, we first record the following coupled decay estimate for the three-dimensional Hall-MHD system, which is a slight extension of \cite[Theorem 1.3]{BJS}. Since the underlying argument is analogous, we omit its proof.
\begin{proposition}\upshape \label{prop:coupled_decay}
Let $d=3$, and let $(u,B)$ be a global Leray-Hopf weak solution of \eqref{HMHD} corresponding to $(u_0,B_0)\in (L^2_\sigma\cap\mathcal{Y}^{\sigma_1}_l)(\R^3)\times (L^2_\sigma\cap\mathcal{Y}^{\sigma_2}_l)(\R^3)$ for $\sigma_1,\sigma_2\in[-1,\frac{3}{2})$. Then, $(u,B)$ decays as
\[
\|u(t)\|_{L^2}^2+\|B(t)\|_{L^2}^2\le C(1+t)^{-\frac{3}{2}+\max{\{\sigma_1,\sigma_2\}}},\quad\text{for all $t\ge0$.}
\]
\end{proposition}
For Leray-Hopf weak solutions, the energy inequality provides only control of the combined energy
$\|u(t)\|_{L^2}^2+\|B(t)\|_{L^2}^2$. Accordingly, the above proposition yields a coupled decay estimate, but it does not distinguish the decay mechanisms of the velocity and magnetic fields. To obtain decoupled decay
rates, one needs enough regularity to apply the Fourier splitting argument separately to the two equations and to control the nonlinear terms at that level.

This is precisely where eventual smoothness becomes useful in two dimensions. By Theorem \ref{thm:2D-eventual-regularity}, every two-dimensional Leray--Hopf weak solution becomes smooth after some finite time. Once this
regularity is available, the Fourier splitting argument can be applied separately to the velocity and magnetic equations, yielding the following decoupled decay rates.

\begin{theorem}\label{thm:2D-decay}\upshape
Let $d=2$, and let $(u,B)$ be a global Leray-Hopf weak solution of \eqref{HMHD} corresponding to $(u_0,B_0)\in L^2_\sigma(\R^2)$. Then, $(u,B)$ decays as follows provided that the initial data meet the additional condition in each case. 
\begin{enumerate}
\item If $B_0\in \mathcal{Y}^{\sigma_2}_l(\R^2)$ for $\sigma_2\in(-1,1)$, then 
\[
\|B(t)\|_{L^2}^2\le C(1+t)^{-1+\sigma_2},\quad\text{for all $t\ge0$.}
\]
\item If $(u_0,B_0)\in \mathcal{Y}^{\sigma_1}_l(\R^2)\times\mathcal{Y}^{\sigma_2}_l(\R^2)$ for $(\sigma_1,\sigma_2)\in [-1,1)\times [-1,1)$ except for the pair $(\sigma_1,\sigma_2)=(-1,0)$, then 
\[
\|u(t)\|_{L^2}^2\leq C(1+t)^{-1+\max{\{\sigma_1,-1+2\sigma_2\}}},\quad \|B(t)\|_{L^2}^2\leq C(1+t)^{-1+\sigma_2},\quad \text{for all $t\ge0$.}
\]
If $(u_0,B_0)\in \mathcal{Y}^{-1}_l(\R^2)\times\mathcal{Y}^{0}_l(\R^2)$, then 
\[
\|u(t)\|_{L^2}^2\leq C(1+t)^{-2}\lt(\log(e+t)\rt)^2,\quad \|B(t)\|_{L^2}^2\leq C(1+t)^{-1},\quad \text{for all $t\ge0$.}
\]
\end{enumerate}
\end{theorem}

Theorem \ref{thm:2D-decay} exhibits an asymmetry between the magnetic and velocity fields. Part (1) shows that a low-frequency assumption on $B_0$ alone yields an explicit algebraic decay rate for the magnetic field. No corresponding low-frequency assumption on $u_0$ is needed for this estimate. This reflects the structure of the magnetic equation \eqref{2.5Hall2}: $u$ appears only linearly in $\widetilde{\nabla}\times(u\times B)$, and the energy bound of $u$ is sufficient for the Fourier splitting argument with respect to $B$ (except for the endpoint case $\sigma_2=-1$).

By contrast, the velocity equation contains the Lorentz force $(\widetilde\nabla\times B)\times B$. Hence the decay of $u$ is affected not only by the low-frequency behavior of $u_0$, but also by the decay already obtained for $B$. This interaction leads to the exponent $\max\{\sigma_1,-1+2\sigma_2\}$ in part (2), with the logarithmic correction appearing at the borderline pair $(\sigma_1,\sigma_2)=(-1,0)$.

In three dimensions, eventual smoothness is not known for arbitrary Leray-Hopf weak solutions of \eqref{HMHD}. Thus, we establish the decoupled decay estimates for global strong solutions, whose existence is known for sufficiently small initial data \cite{Chae Degond Liu, DT22, Tan}. This additional regularity allows us to apply the Fourier splitting argument separately to the velocity and magnetic equations.

\begin{theorem}\upshape \label{thm:3D-decay}
Let $d=3$, and let $(u,B)$ be a global strong solution of \eqref{HMHD} with divergence-free initial data $(u_0,B_0)$. Assume in addition that $(u_0,B_0)\in L^2_\sigma(\R^3)$. Then the following decay estimates hold.
\begin{enumerate}
\item If $B_0\in \mathcal{Y}^{\sigma_2}_l(\R^3)$ for $\sigma_2\in[-1,\frac{3}{2})$, then
\[
\|B(t)\|_{L^2}^2 \le C(1+t)^{-\frac{3}{2}+\sigma_2},\quad\text{for all $t\ge0$.}
\]
\item If $u_0\in\mathcal{Y}^{\sigma_1}_l(\R^3)$ for $\sigma_1\in[1,\frac{3}{2})$, then
\[
\|u(t)\|_{L^2}^2\le C(1+t)^{-\frac{3}{2}+\sigma_1},\quad\text{for all $t\ge0$.}
\]
\item If $(u_0,B_0)\in\mathcal{Y}^{\sigma_1}_l(\R^3)\times \mathcal{Y}^{\sigma_2}_l(\R^3)$ for $\sigma_1\in[-1,1]$, $\sigma_2\in [-1,\frac{3}{2})$ except for the pair $(\sigma_1,\sigma_2)=(-1,\frac12)$, then
\[
\|u(t)\|_{L^2}^2\le C(1+t)^{-\frac{3}{2}+\max{\{\sigma_1,-2+2\sigma_2\}}},\quad\text{for all $t\ge0$.}
\]
If $(u_0,B_0)\in\mathcal{Y}^{-1}_l(\R^3)\times \mathcal{Y}^{\frac{1}{2}}_l(\R^3)$, then
\[
\|u(t)\|_{L^2}^2\le C(1+t)^{-2}\lt(\log{(e+t)}\rt)^2,\quad\text{for all $t\ge0$.}
\]
\end{enumerate}
\end{theorem}

The estimates for $B$ depend only on the low-frequency behavior of $B_0$. For the velocity field, there are two regimes. If $\sigma_1\in[1,\frac32)$, the decay rate of $u$ is sufficiently slow, and no additional low-frequency condition on $B_0$ is needed. For faster velocity decay, namely when $\sigma_1\in[-1,1]$, the Lorentz force may dominate the large-time behavior. In that regime the decay rate is determined by the larger of the two low-frequency contributions, leading to the exponent $\max\{\sigma_1,-2+2\sigma_2\}$. The logarithmic correction occurs at the borderline pair $(\sigma_1,\sigma_2)=(-1,\frac12)$.

\begin{remark}\upshape
We have a list of remarks related to Theorem \ref{thm:2D-decay} and Theorem \ref{thm:3D-decay}.
\begin{enumerate}
\item By setting $B\equiv0$, we derive the decay results of the incompressible Navier-Stokes equations for $d=2, 3$: Let $u$ be a global Leray-Hopf weak solution of the incompressible Navier-Stokes equations corresponding to $u_0\in L^2_\sigma(\R^d)\cap\mathcal{Y}^\sigma_l(\R^d)$ for $\sigma\in [-1,\frac{d}{2})$. Then, $u$ decays as
\[
\|u(t)\|_{L^2}^2\le C(1+t)^{-\frac{d}{2}+\sigma},\quad\text{for all $t\ge0$.}
\]
\item If we consider Leray-Hopf weak solutions for the classical MHD system, the solutions attain the eventual smoothness in three dimensions as well as two dimensions. Therefore, we can get the same decay results with Theorems \ref{thm:2D-decay} and \ref{thm:3D-decay} for any global Leray-Hopf weak solutions of the MHD system. To the best of our knowledge, such decoupled decay estimates for $u$ and $B$ are new not only for the Hall-MHD system but also for the classical MHD system.
\end{enumerate}
\end{remark}

\textbf{Structure of the paper.} In the following sections, we will present the proofs of our main theorems in order. In Section \ref{sec:2D-eventual-regularity}, we prove the eventual smoothness for $d=2$ (Theorem \ref{thm:2D-eventual-regularity}). In Section \ref{sec:3D_weak_sol}, we present the existence of 3D weak solutions with eventual regularity of $B+\omega$ (Theorem \ref{thm:3D_weak_sol}). In Section \ref{sec:decay-rates}, we provide the decay rates of the velocity field and magnetic field to \eqref{HMHD} based on the generalized Fourier splitting method for $d=2$ (Theorem \ref{thm:2D-decay}) and $d=3$ (Theorem \ref{thm:3D-decay}). In Appendix \ref{sec:app.C}, we derive the integral formulation for Leray-Hopf weak solutions. Finally, Appendices \ref{sec:app.A} and \ref{sec:app.B} are devoted to the proofs of Propositions \ref{prop:non-uniform} and \ref{prop:non-uniform_B+omega}, respectively.

\section{Eventual smoothness in two dimensions} \label{sec:2D-eventual-regularity}

In this section, we prove Theorem \ref{thm:2D-eventual-regularity}, the eventual regularity of Leray-Hopf weak solutions of \eqref{HMHD} for $d=2$. To this end, we present two propositions as follows.

\begin{proposition} \label{prop:small-data-gwp}
Let $(u_0,B_0)\in L^2_\sigma(\R^2)\cap H^1(\mathbb{R}^2)$. Then, there exists $c_0>0$ depending only on $\|u_0\|_{L^2}$, $\|B_0\|_{L^2}$, $\|B_0+\omega_0\|_{L^2}$, $\nu$, $\eta$ such that if $\|\nabla B_0\|_{L^2}\le c_0$ then there exists a global unique strong solution $(u,B)\in C([0,\infty);H^1(\R^2))$ of \eqref{HMHD} for $d=2$ with $(\nabla u,\nabla B)\in L^2(0,\infty; H^1(\R^2))$.
\end{proposition}

\begin{proposition} \label{prop:weak-strong}
Let $(u_1,B_1)$ and $(u_2,B_2)$ be two Leray-Hopf weak solutions of \eqref{HMHD} for $d=2$ corresponding to the same initial data $(u_0,B_0)\in L^2_\sigma$. If $\nabla B_2\in L^2(0,T;H^1(\mathbb{R}^2))$ for some $T>0$ (possibly $T=\infty$), then $(u_1,B_1)\equiv (u_2,B_2)$ on $[0,T)$.
\end{proposition}

Proposition \ref{prop:small-data-gwp} represents the small data global well-posedness of strong solutions to \eqref{HMHD} for $d=2$, which updates the result in \cite{Tan} based on the uniform bound obtained in \cite{BKS} (see \eqref{bound_omega}). Proposition \ref{prop:weak-strong} represents the weak-strong uniqueness for \eqref{HMHD} with $d=2$. The weak-strong uniqueness has been studied in some literature; \cite{BK} for $d=2$, \cite{Chae Degond Liu} for $d=3$. However, we prove it for our specific use to prove Theorem \ref{thm:2D-eventual-regularity}.

\begin{proof}[Proof of Theorem \ref{thm:2D-eventual-regularity}]
Let $(u,B)$ be a global Leray-Hopf weak solution. Then, from \eqref{energy ineq}, we have that for all $t\geq0$
\[
\|u(t)\|_{L^2}^2+\|B(t)\|_{L^2}^2+2\nu\int^t_0\|\nabla u(\tau)\|_{L^2}^2\,d\tau+2\eta\int^t_0\|\nabla B(\tau)\|_{L^2}^2\,d\tau\leq \|u_0\|_{L^2}^2+\|B_0\|_{L^2}^2.
\]
Hence, there exists $T_0>0$ satisfying
\[
\|u(T_0)\|_{L^2}^2+\|B(T_0)\|_{L^2}^2\leq \|u_0\|_{L^2}^2+\|B_0\|_{L^2}^2,\quad \|\nabla B(T_0)\|_{L^2}^2\leq c_0
\]
for the constant $c_0>0$ in Proposition \ref{prop:small-data-gwp}. Then, there exists a global strong solution $(\widebar{u},\widebar{B})\in C([T_0,\infty);H^1(\mathbb{R}^2))$ with $(\nabla\widebar{u},\nabla\widebar{B})\in L^2(T_0,\infty;H^1(\mathbb{R}^2))$ starting from $(u(T_0),B(T_0))$ at $t=T_0$. We deduce from Proposition \ref{prop:weak-strong} that $(u,B)\equiv(\widebar{u},\widebar{B})$ on $[T_0,\infty)$. This completes the proof.
\end{proof}

Before proving Proposition \ref{prop:small-data-gwp}, we first recall the uniform bounds for \eqref{HMHD} for $d=2$ obtained in \cite[Theorem 1.2]{BKS}:
\[
\|(B+\omega)(t)\|_{L^2}^2+\nu\int^t_0\|\nabla(B+\omega)(\tau)\|_{L^2}^2\,d\tau\leq \left(\|B_0+\omega_0\|_{L^2}^2+C\mathcal{E}_0 \frac{(\nu-\eta)^2}{\nu\eta}\right) e^{C\nu^{-2}\mathcal{E}_0}=:\mathcal{E}_1,
\]
where $\mathcal{E}_0$ is defined in \eqref{weak_en_ineq} and
\begin{equation}\label{bound_omega}
\|\omega(t)\|_{L^2}^2+\min{(\nu,\eta)}\int^t_0\|\nabla\omega(\tau)\|_{L^2}^2\,d\tau\leq \mathcal{E}_0+\mathcal{E}_1=:\mathcal{E}_2
\end{equation}
In the help of the regularity gain of $u$ in \eqref{bound_omega}, we deal with \eqref{HMHD} as like \eqref{EMHD} with mild forcing terms.

\subsection{Proof of Proposition \ref{prop:small-data-gwp}}

Since the local well-posedness of strong solutions in $H^1(\R^2)$ was shown in \cite{BK, Tan}, it suffices to check \textit{a priori estimates} implying the global continuation of the existence time. As we mentioned before, the solutions satisfy the energy inequality \eqref{energy ineq} (energy identity, in fact) and the uniform bound with respect to $\omega$ \eqref{bound_omega}.

We write $J:= \widetilde{\nabla}\times B$ and obtain
\[\begin{aligned}
\frac{1}{2}\frac{d}{dt}\|\nabla B\|_{L^2}^2+\eta\|\Delta B\|_{L^2}^2&=\int_{\mathbb{R}^2} \Delta B\cdot (u\cdot\widetilde{\nabla} B-B\cdot\widetilde{\nabla} u)\,dx+\int_{\mathbb{R}^2}\Delta B\cdot\widetilde{\nabla}\times(J\times B)\,dx \\
&=: \text{I}_1+\text{I}_2+\text{I}_3.
\end{aligned}\]
We separately estimate $\text{I}_1, \text{I}_2, \text{I}_3$ as
\[\begin{aligned}
\text{I}_1&=-\sum_{k=1}^2 \int_{\mathbb{R}^2}\partial_k B\cdot(\partial_k u\cdot\widetilde{\nabla} B)\,dx\leq \|\nabla u\|_{L^2}\|\nabla B\|_{L^4}^2\leq C\|\nabla u\|_{L^2}\|\nabla B\|_{L^2}\|\Delta B\|_{L^2},\\
\text{I}_{2}&\leq \|\Delta B\|_{L^2}\|B\|_{L^4}\|\nabla u\|_{L^4}\leq C\|\nabla u\|_{L^4}\|\Lambda^{\frac{1}{2}}B\|_{L^2}\|\Delta B\|_{L^2},\\
\text{I}_3&=-\sum_{k=1}^2\int_{\mathbb{R}^2} \partial_k J\cdot (J\times\partial_k B)\,dx\leq C\|\nabla B\|_{L^4}^2\|\Delta B\|_{L^2}\leq C\|\nabla B\|_{L^2}\|\Delta B\|_{L^2}^2.
\end{aligned}\]
Due to $\text{I}_2$, we need to control $\|\Lambda^{\frac{1}{2}}B\|_{L^2}$. So we also estimate

\[
\begin{aligned}
\frac{1}{2}\frac{d}{dt}\|\Lambda^{\frac{1}{2}}B\|_{L^2}^2+\eta\|\Lambda^{\frac{3}{2}}B\|_{L^2}^2&=-\int_{\mathbb{R}^2}\Lambda^{\frac12} B\cdot \Lambda^{\frac12}(u\cdot\widetilde{\nabla} B - B \cdot \widetilde{\nabla} u)\,dx\\
&\quad-\int_{\mathbb{R}^2}\Lambda^{\frac12} B\cdot \Lambda^{\frac12}\lt[\widetilde{\nabla}\times(J\times B)\rt]\,dx\\
&=:\text{I}_4+\text{I}_5+\text{I}_6.
\end{aligned}
\]
For $\text{I}_4$, we have
\[\begin{aligned}
\text{I}_4&=-\int_{\mathbb{R}^2}\Lambda^{\frac{1}{2}}B\cdot\left[\Lambda^{\frac{1}{2}}(u\cdot\widetilde{\nabla} B)-u\cdot\widetilde{\nabla}\Lambda^{\frac{1}{2}}B\right]\,dx\\
&\leq \|\Lambda^{\frac{1}{2}}B\|_{L^2}\|\Lambda^{\frac{1}{2}}u\|_{L^4}\|\nabla B\|_{L^4} \leq C\|\nabla u\|_{L^2}\|\Lambda^{\frac{1}{2}}B\|_{L^2}\|\Lambda^{\frac{3}{2}}B\|_{L^2},
\end{aligned}
\]
and for $\text{I}_5$,

\[\begin{aligned}
\text{I}_5&= -\int_{\R^2}\Lambda B \cdot (B\cdot\widetilde{\nabla} u)\,dx \leq \|\Lambda B\|_{L^4}\|B\|_{L^4}\|\nabla u\|_{L^2}\leq C\|\nabla u\|_{L^2}\|\Lambda^{\frac{1}{2}}B\|_{L^2}\|\Lambda^{\frac{3}{2}}B\|_{L^2},
\end{aligned}\]
Next, for $\text{I}_6$, 

\[\begin{aligned}
\text{I}_6&=-\int_{\mathbb{R}^2}\Lambda^{\frac{1}{2}}J\cdot\left[\Lambda^{\frac{1}{2}}(J\times B)-(\Lambda^{\frac{1}{2}}J)\times B\right]\,dx\\
&\leq \|\Lambda^{\frac{1}{2}}J\|_{L^2}\|J\|_{L^4}\|\Lambda^{\frac{1}{2}}B\|_{L^4}\leq C\|\nabla B\|_{L^2}\|\Lambda^{\frac{3}{2}}B\|_{L^2}^2.
\end{aligned}\]

Now, we set
\[
X^2(t):=\|\Lambda^{\frac{1}{2}}B(t)\|_{L^2}^2+\|\nabla B(t)\|_{L^2}^2,\quad \mbox{and} \quad Y^2(t):=\|\Lambda^{\frac{3}{2}}B(t)\|_{L^2}^2+\|\Delta B(t)\|_{L^2}^2,
\]
and collect the above estimates to obtain
\[\begin{aligned}
\frac{d}{dt}X^2+2\eta Y^2 &\leq C\left(\|\nabla u\|_{L^2}+\|\nabla u\|_{L^4}\right)XY+C\|\nabla B\|_{L^2}Y^2 \\
&\leq \left(\frac{\eta}{2}+C_0\|\nabla B\|_{L^2}\right)Y^2+\widetilde{C}_0\eta^{-1}\left(\|\nabla u\|_{L^2}^2+\|\nabla u\|_{L^4}^2\right)X^2
\end{aligned}\]
for some $C_0,\widetilde{C}_0>0$. Suppose that $\displaystyle\sup_{0\leq t\leq t_0}\|\nabla B(t)\|_{L^2}\leq \frac{\eta}{2C_0}$. Then we use  the Gr\"onwall's inequality to attain
\[
X^2(t)+\eta\int^t_0 Y^2(\tau)\,d\tau\leq X^2(0)\exp\left(\widetilde{C}_0\eta^{-1}\int^t_0(\|\nabla u(\tau)\|_{L^2}^2+\|\nabla u(\tau)\|_{L^4}^2)\,d\tau\right)
\]
From \eqref{bound_omega} we deduce that if 
\[
X^2(0)\exp\left(\frac{2\widetilde{C}_0\mathcal{E}_2}{\eta\min{(\nu,\eta)}}\right)\leq \frac{\eta^2}{4C_0^2}
\]
$X^2(t)+\eta\int^t_0 Y^2(\tau)\,d\tau$ has the uniform-in-time bound, and so the global well-posedness follows by the blow up criteria for the 2D Hall-MHD system (see \cite{BKS, Chae Lee}). Finally, we use \eqref{energy ineq} and the Sobolev interpolation to complete the proof of Proposition \ref{prop:small-data-gwp}.

\subsection{Proof of Proposition \ref{prop:weak-strong}}

Let $(u_1,B_1)$ and $(u_2,B_2)$ be two such weak solutions and define $(\widetilde{u},\widetilde{B})=(u_1-u_2,B_1-B_2)$. Then we employ the relative energy method (see e.g. \cite[Section 4.4]{BV_text}). From the energy inequality, we deduce that
\[\begin{aligned}
& \frac{1}{2}\|\widetilde{u}(t)\|_{L^2}^2+\frac{1}{2}\|\widetilde{B}(t)\|_{L^2}^2+\nu\int^t_0\|\nabla\widetilde{u}(\tau)\|_{L^2}^2\,d\tau+\eta\int^t_0\|\nabla \widetilde{B}(\tau)\|_{L^2}^2\,d\tau \\
& = \sum_{i=1,2}\left[\frac{1}{2}\|u_i(t)\|_{L^2}^2+\frac{1}{2}\|{B}_i(t)\|_{L^2}^2+\nu\int^t_0\|\nabla u_i(\tau)\|_{L^2}^2\,d\tau+\eta\int^t_0\|\nabla {B}_i(\tau)\|_{L^2}^2\,d\tau\right] \\
&-\int_{\mathbb{R}^2} u_1(t,x)\cdot u_2(t,x)\,dx -\int_{\mathbb{R}^2} B_1(t,x)\cdot B_2(t,x)\,dx \\
& -2\nu\int^t_0\int_{\mathbb{R}^2}\nabla u_1(\tau,x):\nabla u_2(\tau,x)\,dxd\tau -2\eta\int^t_0\int_{\mathbb{R}^2}\nabla B_1(\tau,x):\nabla B_2(\tau,x)\,dxd\tau \\
&\leq \|u_0\|_{L^2}^2 -\int_{\mathbb{R}^2} u_1(t,x)\cdot u_2(t,x)\,dx -2\nu\int^t_0\int_{\mathbb{R}^2}\nabla u_1(\tau,x):\nabla u_2(\tau,x)\,dxd\tau \\
&+ \|B_0\|_{L^2}^2 -\int_{\mathbb{R}^2} B_1(t,x)\cdot B_2(t,x)\,dx -2\eta\int^t_0\int_{\mathbb{R}^2}\nabla B_1(\tau,x):\nabla B_2(\tau,x)\,dxd\tau.
\end{aligned}\]
Now,  the fundamental theorem of calculus gives
\[\begin{aligned}
\|u_0\|_{L^2}^2 &-\int_{\mathbb{R}^2} u_1(t,x)\cdot u_2(t,x)\,dx=-\int^t_0\frac{d}{d\tau}\left(\int_{\mathbb{R}^2}u_1(\tau,x)\cdot u_2(\tau,x)\,dx\right)\,d\tau \\
&=-\int^t_0\int_{\mathbb{R}^2}\lt(\partial_\tau u_1(\tau,x)\cdot u_2(\tau,x)+\partial_\tau u_2(\tau,x)\cdot u_1(\tau,x)\rt)\,dxd\tau \\
&=2\nu\int^t_0\int_{\mathbb{R}^2}\nabla u_1:\nabla u_2\,dxd\tau +\text{II}_1+\text{II}_2
\end{aligned}\]
where

\[\begin{aligned}
\text{II}_1&=\int^t_0\int_{\mathbb{R}^2} \lt[ (u_1\cdot\widetilde{\nabla} u_1)\cdot u_2+(u_2\cdot\widetilde{\nabla} u_2)\cdot u_1\rt]\,dxd\tau, \\
\text{II}_2&=-\int^t_0\int_{\mathbb{R}^2} \lt[(B_1\cdot\widetilde{\nabla} B_1)\cdot u_2+(B_2\cdot\widetilde{\nabla} B_2)\cdot u_1\rt]\,dxd\tau.
\end{aligned}\]
Here, we tested the equations of $u_1$ and $u_2$ by $u_2$ and $u_1$ since $u_i\in L^2([0,\infty);\dot{H}^1(\mathbb{R}^2))$ solve the equations in $L^2_{loc}([0,\infty);\dot{H}^{-1}(\mathbb{R}^2))$ from the energy inequality. Similarly, we also obtain
\[\begin{aligned}
\|B_0\|_{L^2}^2 &-\int_{\mathbb{R}^2} B_1(t,x)\cdot B_2(t,x)\,dx\\
&=-\int^t_0\frac{d}{d\tau}\left(\int_{\mathbb{R}^2}B_1(\tau,x)\cdot B_2(\tau,x)\,dx\right)\,d\tau \\
&=-\int^t_0\int_{\mathbb{R}^2}\lt(\partial_\tau B_1(\tau,x)\cdot B_2(\tau,x)+\partial_\tau B_2(\tau,x)\cdot B_1(\tau,x)\rt)\,dxd\tau \\
&=2\eta\int^t_0\int_{\mathbb{R}^2}\nabla B_1:\nabla B_2\,dxd\tau +\text{II}_3+\text{II}_4+\text{II}_5
\end{aligned}\]
where
\[\begin{aligned}
\text{II}_3&=\int^t_0\int_{\mathbb{R}^2}\lt[ (u_1\cdot\widetilde{\nabla} B_1)\cdot B_2+(u_2\cdot\widetilde{\nabla} B_2)\cdot B_1\rt]\,dxd\tau,\\
\text{II}_4&=-\int^t_0\int_{\mathbb{R}^2}\lt[ (B_1\cdot\widetilde{\nabla} u_1)\cdot B_2+(B_2\cdot\widetilde{\nabla} u_2)\cdot B_1\rt]\,dxd\tau,\\
\text{II}_5&=\int^t_0\int_{\mathbb{R}^2}\lt[((\widetilde{\nabla}\times B_1)\times B_1)\cdot(\widetilde{\nabla}\times B_2)+((\widetilde{\nabla}\times B_2)\times B_2)\cdot(\widetilde{\nabla}\times B_1)\rt]\,dxd\tau.
\end{aligned}\]
Note that although $\widetilde{\nabla}\times((\widetilde{\nabla}\times B_1)\times B_1)$ is in $L^2_{loc}([0,\infty);\dot{H}^{-2}(\mathbb{R}^2))$, $\text{II}_5$ is valid since $B_2\in L^2([0,T);\dot{H}^2(\mathbb{R}^2))$. We collect all the estimates to arrive at
\[
\frac{1}{2}\|\widetilde{u}(t)\|_{L^2}^2+\frac{1}{2}\|\widetilde{B}(t)\|_{L^2}^2+\nu\int^t_0\|\nabla\widetilde{u}(\tau)\|_{L^2}^2\,d\tau+\eta\int^t_0\|\nabla \widetilde{B}(\tau)\|_{L^2}^2\,d\tau\leq \text{II}_1+\text{II}_2+\text{II}_3+\text{II}_4+\text{II}_5.
\]
For $\text{II}_1$,
\[
\text{II}_1=-\int^t_0\int_{\mathbb{R}^2}(u_1\cdot\widetilde{\nabla} u_1-u_2\cdot\widetilde{\nabla} u_2)\cdot\widetilde{u}\,dxd\tau=-\int^t_0\int_{\mathbb{R}^2}(\widetilde{u}\cdot\widetilde{\nabla} u_2)\cdot\widetilde{u}\,dxd\tau,
\]
and similarly we obtain 
\[
\text{II}_3=-\int^t_0\int_{\mathbb{R}^2}(\widetilde{u}\cdot\widetilde{\nabla} B_2)\cdot\widetilde{B}\,dxd\tau.
\]
For $\text{II}_2$ and $\text{II}_4$, we use $\langle B\cdot\widetilde{\nabla} F,G\rangle+\langle B\cdot\widetilde{\nabla} G,F\rangle=0$ to get
\[\begin{aligned}
\text{II}_2+\text{II}_4 &=\int^t_0\int_{\mathbb{R}^2}\lt[(B_1\cdot\widetilde{\nabla} B_1-B_2\cdot\widetilde{\nabla} B_2)\cdot\widetilde{u}+(B_1\cdot\widetilde{\nabla} u_1-B_2\cdot\widetilde{\nabla} u_2)\cdot\widetilde{B}\rt]\,dxd\tau \\
&=\int^t_0\int_{\mathbb{R}^2}\lt[(\widetilde{B}\cdot\widetilde{\nabla} B_2)\cdot\widetilde{u}+(\widetilde{B}\cdot\widetilde{\nabla} u_2)\cdot\widetilde{B}\rt]\,dxd\tau
\end{aligned}\]
Hence, we compute
\[\begin{aligned}
\text{II}_1+\text{II}_2+\text{II}_3+\text{II}_4 &\leq C\int_0^t \left(\|\nabla u_2\|_{L^2}+\|\nabla B_2\|_{L^2}\right)\left(\|\widetilde{u}\|_{L^4}^2+\|\widetilde{B}\|_{L^4}^2\right)\,d\tau\\
&\leq \frac{\nu}{2}\int_0^t \|\nabla\widetilde{u}\|_{L^2}^2\,d\tau+\frac{\eta}{4}\int_0^t \|\nabla\widetilde{B}\|_{L^2}^2\,d\tau\\
&\quad +\frac{C}{\min{(\nu,\eta)}}\int_0^t \left(\|\nabla u_2\|_{L^2}^2+\|\nabla B_2\|_{L^2}^2\right)\left(\|\widetilde{u}\|_{L^2}^2+\|\widetilde{B}\|_{L^2}^2\right)\,d\tau.
\end{aligned}\]
For $\text{II}_5$, 
\[\begin{aligned}
\text{II}_5 &=-\int^t_0\int_{\mathbb{R}^2}\lt[\left((\widetilde{\nabla}\times B_1)\times B_1-(\widetilde{\nabla}\times B_2)\times B_2\right)\cdot(\widetilde{\nabla}\times\widetilde{B})\rt]\,dxd\tau \\
&= -\int^t_0\int_{\mathbb{R}^2}((\widetilde{\nabla}\times B_2)\times\widetilde{B})\cdot(\widetilde{\nabla}\times \widetilde{B})\,dxd\tau\\
&=\int^t_0\int_{\mathbb{R}^2}(\widetilde{\nabla}\times B_2)\cdot((\widetilde{\nabla}\times \widetilde{B})\times\widetilde{B})\,dxd\tau \\
&=-\int^t_0\int_{\mathbb{R}^2}\widetilde{\nabla}(\widetilde{\nabla}\times B_2):(\widetilde{B}\otimes\widetilde{B})\,dxd\tau\\
&\leq C\int_0^t \|\Delta B_2\|_{L^2}\|\widetilde{B}\|_{L^4}^2\,d\tau\\
&\leq \frac{\eta}{4}\int_0^t \|\nabla\widetilde{B}\|_{L^2}^2\,d\tau+C\eta^{-1}\int_0^t \|\Delta B_2\|_{L^2}^2\|\widetilde{B}\|_{L^2}^2\,d\tau
\end{aligned}\]
Therefore, we have
\[\begin{aligned}
\|\widetilde{u}(t)\|_{L^2}^2+\|\widetilde{B}(t)\|_{L^2}^2 &+\nu\int^t_0\|\nabla\widetilde{u}(\tau)\|_{L^2}^2\,d\tau+\eta\int^t_0\|\nabla \widetilde{B}(\tau)\|_{L^2}^2\,d\tau \\
&\leq \frac{C}{\min{(\nu,\eta)}}\int_0^t \left(\|\nabla u_2\|_{L^2}^2+\|\nabla B_2\|_{L^2}^2\right)\left(\|\widetilde{u}\|_{L^2}^2+\|\widetilde{B}\|_{L^2}^2\right)\,d\tau\\
&\quad +C\eta^{-1}\int_0^t \|\Delta B_2\|_{L^2}^2\|\widetilde{B}\|_{L^2}^2\,d\tau,
\end{aligned}\]
and the weak-strong uniqueness follows from Gr\"onwall's lemma.

\section{Eventual regularity of $B+\omega$ in three dimensions}\label{sec:3D_weak_sol}

In this section, we prove Theorem \ref{thm:3D_weak_sol}, the existence of Leray-Hopf weak solutions with eventual regularity of $B+\omega$. To this end, we present three propositions as follows:

\begin{proposition} \label{prop:existence_B+omega}
\begin{enumerate}
\item Let $\nu,\eta>0$. If $(u_0,B_0,\omega_0)\in L^2_\sigma(\R^3)$ and 
\[
C_0:=\lt(\|B_0+\omega_0\|_{L^2}^2+C\mathcal{E}_0\frac{(\nu-\eta)^2}{\nu\eta}\rt)\exp{\lt(\frac{C}{\nu^4}(\mathcal{E}_0+\mathcal{E}_0^2)\rt)}<1,\quad \mathcal{E}_0:=\|u_0\|_{L^2}^2+\|B_0\|_{L^2}^2,
\]
then there exists a global Leray-Hopf weak solution $(u,B)$ of \eqref{HMHD} satisfying 
\begin{equation}\label{3D_B+omega_bound}
\|(B+\omega)(t)\|_{L^2}^2+\nu\int^t_0\|\nabla(B+\omega)(\tau)\|_{L^2}^2\,d\tau\le C_0
\end{equation}
for all $t\ge0$.
\item Let $\nu=\eta>0$. If $(u_0,B_0)\in L^2_\sigma(\R^3)$, $B_0+\omega_0\in H^k(\R^3)$ for $k=1, 2$ and
\[
C_0=\|B_0+\omega_0\|_{L^2}^2\exp{\lt(\frac{C}{\nu^4}(\mathcal{E}_0+\mathcal{E}_0^2)\rt)}<1,
\]
then there exists a global Leray-Hopf weak solution $(u,B)$ of \eqref{HMHD} satisfying 
\begin{equation}\label{3D_B+omega_bound2}
\|(B+\omega)(t)\|_{H^k}^2 +\nu\int^t_0\|\nabla(B+\omega)(\tau)\|_{H^k}^2\,d\tau\le C(\nu,\mathcal{E}_0, \|B_0+\omega\|_{H^k})
\end{equation}
for all $t\ge 0$.
\end{enumerate}
\end{proposition}

We shall also use the following asymptotic behavior of Leray-Hopf weak solutions, both in the proof of Theorem \ref{thm:3D_weak_sol} and in the subsequent derivation of decay rates.

\begin{proposition}\label{prop:non-uniform}
Let $d=2,3$, and let $(u,B)$ be a global Leray-Hopf weak solution of \eqref{HMHD} corresponding to $(u_0,B_0)\in L^2_\sigma(\R^d)$. Then, 
\[
\|u(t)\|_{L^2}^2+\|B(t)\|_{L^2}^2\rightarrow0\quad\text{as $t\rightarrow\infty$.}
\]
\end{proposition}

\begin{proposition}\label{prop:non-uniform_B+omega}
Let $d=3$, and let $(u,B)$ be a global Leray-Hopf weak solution of \eqref{HMHD} constructed in Proposition \ref{prop:existence_B+omega} (1). Then, 
\[
\|(B+\omega)(t)\|_{L^2}\rightarrow 0\quad \text{as $t\rightarrow\infty$.}
\]
\end{proposition}

Proposition \ref{prop:existence_B+omega} represents that under some assumptions on initial data $(u_0,B_0)$ there exists a global Leray-Hopf weak solution having uniform-in-time bounds of $B+\omega$. Proposition \ref{prop:non-uniform} represents the $L^2$ decay of $u$ and $B$ for Leray-Hopf weak solutions, and Proposition \ref{prop:non-uniform_B+omega} represents the $L^2$ decay of $B+\omega$ for the weak solutions considered in Proposition \ref{prop:existence_B+omega}. We first prove Theorem \ref{thm:3D_weak_sol} assuming Propositions \ref{prop:existence_B+omega}-\ref{prop:non-uniform_B+omega} and then prove Proposition \ref{prop:existence_B+omega}. For smooth flow of reading, we postpone the proofs of Propositions \ref{prop:non-uniform} and \ref{prop:non-uniform_B+omega} to Appendices \ref{sec:app.A} and \ref{sec:app.B}, respectively.

\begin{proof}[Proof of Theorem \ref{thm:3D_weak_sol}]

From the energy inequality \eqref{energy ineq} and Proposition \ref{prop:non-uniform}, a Leray-Hopf weak solution of \eqref{HMHD} satisfies that for any $\epsilon>0$ there exists $T_0>0$ such that 
\[
\mathcal{E}_{T_0}:=\|u(T_0)\|_{L^2}^2+\|B(T_0)\|_{L^2}^2\le \epsilon,\quad \|\omega(T_0)\|_{L^2}^2=\|\nabla u(T_0)\|_{L^2}^2\le \epsilon.
\] 
We choose sufficiently small $\epsilon>0$ to satisfy
\bq\label{CT0_def}
C_{T_0}:=\lt(\|(B+\omega)(T_0)\|_{L^2}^2+C\mathcal{E}_{T_0}\frac{(\nu-\eta)^2}{\nu\eta}\rt)\exp{\lt(\frac{C}{\nu^4}(\mathcal{E}_{T_0}+\mathcal{E}_{T_0}^2)\rt)}<1.
\eq
Then, from Proposition \ref{prop:existence_B+omega} we can find a Leray-Hopf weak solution $(\tilde{u},\tilde{B})$ on $[T_0,\infty)$, which satisfies 
\[
\|(\tilde{B}+\nabla\times\tilde{u})(t)\|_{L^2}^2+\nu\int^t_{T_0}\|\nabla(\tilde{B}+\nabla\times\tilde{u})(\tau)\|_{L^2}^2\,d\tau\le C_{T_0},\quad\text{for all $t\ge T_0$.}
\]
Moreover, from Proposition \ref{prop:non-uniform_B+omega}, $\|(B+\omega)(t)\|_{L^2}\rightarrow 0$ as $t\rightarrow \infty$ as well. Hence, we conclude the proof of (1) of Theorem \ref{thm:3D_weak_sol} by defining a new solution (still denoted by $(u,B)$)
\[
\begin{cases}
(u,B)\quad \text{on $[0,T_0]$,}\\
(\tilde{u},\tilde{B})\quad \text{on $[T_0,\infty)$.}
\end{cases}
\]

Additionally if $\nu=\eta$, from the inequality 
\[
\|(B+\omega)(t)\|_{L^2}^2+\nu\int^t_{T_0}\|\nabla(B+\omega)(\tau)\|_{L^2}^2\,d\tau\le C_{T_0},\quad t\ge T_0,
\]
we have $T_1 \ge T_0$ such that $(u(T_1),B(T_1))\in L^2(\R^3)$ and $(B+\omega)(T_1)\in H^1(\R^3)$ with $C_{T_1}<1$, where $C_{T_1}$ is simply obtained by replacing $T_0$ with $T_1$ in \eqref{CT0_def}. Then by (2) of Proposition \ref{prop:existence_B+omega}, there exists a Leray-Hopf weak solution $(u_1,B_1)$ on $[T_1,\infty)$ such that $B_1+\nabla\times u_1\in L^\infty([T_1,\infty);H^1)$ and $\nabla(B_1+\nabla\times u_1)\in L^2(T_1,\infty;H^1)$. In a similar manner, there exists $T_2\ge T_1$ such that $(u_1(T_2),B_1(T_2))\in L^2(\R^3)$ and $(B_1+\nabla\times u_1)(T_2)\in H^2(\R^3)$ with $C_{T_2}<1$, which implies that there exists a Leray-Hopf weak solution $(u_2,B_2)$ on $[T_2,\infty)$ such that $B_2+\nabla\times u_2\in L^\infty([T_2,\infty);H^2)$ and $\nabla(B_2+\nabla\times u_2)\in L^2(T_2,\infty;H^2)$. Since the desired decay rates of (2) follow by the uniform bounds \eqref{3D_B+omega_bound2} and the Sobolev embedding theory, we conclude the proof of (2) of Theorem \ref{thm:3D_weak_sol}.
\end{proof}

\subsection{Proof of Proposition \ref{prop:existence_B+omega}}
The uniform bound \eqref{3D_B+omega_bound} stems from the a priori estimate similar to \cite[Corollary 1.1]{BKS}. However in contrast to \cite[Corollary 1.1]{BKS}, we do not suppose the condition of coefficients $\frac{|\nu-\eta|}{\nu+\eta}\ll 1$. To this end, we consider a smooth solution $(u,B)$ of \eqref{HMHD} for $d=3$, which naturally satisfies
\[
\frac{1}{2}\frac{d}{dt}\|B+\omega\|_{L^2}^2+\nu\|\nabla(B+\omega)\|_{L^2}^2=(\nu-\eta)\int_{\R^3}\nabla B :\nabla(B+\omega)\,dx+\int_{\R^3}\lt((B+\omega)\cdot\nabla u\rt)\cdot(B+\omega)\,dx
\]
from \eqref{eq of B+omega}. Following the estimates in the proof of \cite[Corollary 1.1]{BKS}, we deduce that
\[
\frac{d}{dt}\|B+\omega\|_{L^2}^2+\frac{\nu}{2}\lt(3-\|B+\omega\|_{L^2}^{\frac{3}{2}}\rt)\|\nabla(B+\omega)\|_{L^2}^2\le C\frac{(\nu-\eta)^2}{\nu}\|\nabla B\|_{L^2}^2+\frac{C}{\nu^3}\lt(1+\|B\|_{L^2}^2\rt)\|\nabla u\|_{L^2}^2\|B+\omega\|_{L^2}^2,
\]
which implies that if $C_0<1$
\[
\|(B+\omega)(t)\|_{L^2}^2+\nu\int^t_0\|\nabla(B+\omega)(\tau)\|_{L^2}^2\,d\tau\le \lt(\|B_0+\omega_0\|_{L^2}^2+C\mathcal{E}_0\frac{(\nu-\eta)^2}{\nu\eta}\rt)\exp{\lt(\frac{C}{\nu^4}(\mathcal{E}_0+\mathcal{E}_0^2)\rt)}=C_0<1
\]
as long as the solution exists.

Now, it is quite standard to show the existence of Leray-Hopf weak solutions satisfying \eqref{3D_B+omega_bound} and \eqref{3D_B+omega_bound2}. We use the mollifier technique \cite{Majda Bertozzi}, and the regularized system of \eqref{HMHD} described in \cite{Chae Degond Liu}. For the following mollifier operator,
\[
\mathcal{J}_\epsilon f=\rho_\epsilon\ast f,\quad \rho_\epsilon=\epsilon^{-3}\rho(x/\epsilon),\quad \rho: \text{non-negative $C^\infty_0$ function with unit integral}, 
\]
we consider the regularized system as follows:
\begin{align*}
& \pa_t u_\epsilon+\mathcal{J}_\epsilon(\mathcal{J}_\epsilon\omega_\epsilon\times \mathcal{J}_\epsilon u_\epsilon)+\nabla p_\epsilon=\mathcal{J}_\epsilon\lt( (\nabla\times \mathcal{J}_\epsilon B_\epsilon)\times \mathcal{J}_\epsilon B_\epsilon\rt)+\nu\Delta\mathcal{J}_\epsilon^2 u_\epsilon, \\
& \pa_t B_\epsilon -\nabla\times \lt(\mathcal{J}_\epsilon(\mathcal{J}_\epsilon u_\epsilon\times \mathcal{J}_\epsilon B_\epsilon)\rt)+\nabla\times \lt(\mathcal{J}_\epsilon\lt( (\nabla\times \mathcal{J}_\epsilon B_\epsilon)\times \mathcal{J}_\epsilon B_\epsilon\rt)\rt)=\eta\Delta\mathcal{J}_\epsilon^2 B_\epsilon,\\
&\dv u_\epsilon=\dv B_\epsilon =0,
\end{align*}
with initial condition
\[
(u_\epsilon, B_\epsilon)(0,x)=(\mathcal{J}_\epsilon u_0,\mathcal{J}_\epsilon B_0)(x).
\]
As in \cite{Chae Degond Liu}, for each $\epsilon>0$ there exists a global smooth solution $(u_\epsilon, B_\epsilon)$ satisfying the energy inequality: for all $t\ge s\ge 0$
\[
\frac{1}{2}\lt(\|u_\epsilon(t)\|_{L^2}^2+\|B_\epsilon(t)\|_{L^2}^2\rt)+\nu\int^t_s\|\nabla\mathcal{J}_\epsilon u_\epsilon(\tau)\|_{L^2}^2\,d\tau+\eta\int^t_s \|\nabla \mathcal{J}_\epsilon B_\epsilon(\tau)\|_{L^2}^2\,d\tau\le \frac{1}{2}\lt(\|u_\epsilon(s)\|_{L^2}^2+\|B_\epsilon(s)\|_{L^2}^2\rt).
\]
Moreover, using the equation of $B_\epsilon+\omega_\epsilon$ 
\[
\pa_t(B_\epsilon+\omega_\epsilon)+\nabla\times\lt(\mathcal{J}_\epsilon\lt(\mathcal{J}_\epsilon(B_\epsilon+\omega_\epsilon)\times \mathcal{J}_\epsilon u_\epsilon\rt)\rt)=\nu\Delta\mathcal{J}_\epsilon^2\omega_\epsilon+\eta\Delta\mathcal{J}_\epsilon^2 B_\epsilon,
\]
and following the same process for the a priori estimate, we deduce that if $C_0<1$
\begin{align*}
&\|(B_\epsilon+\omega_\epsilon)(t)\|_{L^2}^2+\nu\int^t_0\|\nabla\mathcal{J}_\epsilon (B_\epsilon+\omega_\epsilon)(\tau)\|_{L^2}^2\,d\tau \\
&\le \lt(\|(B_\epsilon+\omega_\epsilon)(0)\|_{L^2}^2+C\frac{(\nu-\eta)^2}{\nu}\int^t_0\|\nabla\mathcal{J}_\epsilon B_\epsilon(\tau)\|_{L^2}^2\,d\tau\rt)\exp{\lt[\frac{C}{\nu^3}\int^t_0(1+\|B_\epsilon(\tau)\|_{L^2}^2)\|\nabla\mathcal{J}_\epsilon u_\epsilon(\tau)\|_{L^2}^2\,d\tau\rt]}.
\end{align*}
Additionally if $\nu=\eta>0$, we also have
\begin{align*}
&\|(B_\epsilon+\omega_\epsilon)(t)\|_{H^k}^2+\nu\int^t_0\|\nabla\mathcal{J}_\epsilon (B_\epsilon+\omega_\epsilon)(\tau)\|_{H^k}^2\,d\tau \\
&\le \|(B_\epsilon+\omega_\epsilon)(0)\|_{H^k}^2 \exp{\lt[\frac{C}{\nu}\int^t_0(\|\nabla\mathcal{J}_\epsilon u(\tau)\|_{L^2}^2+\|\Delta\mathcal{J}_\epsilon u(\tau)\|_{L^2}^2)\,d\tau\rt]}
\end{align*}
(see the proof of \cite[Theorem 1.1]{BKS}).
Finally, by taking the limit $\epsilon\rightarrow 0$ we conclude that there exists a global Leray-Hopf weak solution satisfying the desired uniform-in-time bound of $B+\omega$.

\section{Decoupled decay rates via the generalized Fourier splitting method}\label{sec:decay-rates}

In this section, we prove the decoupled decay rates of Leray-Hopf weak solutions of \eqref{HMHD} applying the generalized Fourier splitting method developed in \cite{BJS} (Theorem \ref{thm:2D-decay} for $d=2$ and Theorem \ref{thm:3D-decay} for $d=3$).
Throughout this section, we set $\nu=\eta=1$ for notational simplicity. It should be noted that the choice of $\nu,\eta>0$ only changes some constants appearing in each estimate, and so do not affect Theorems \ref{thm:2D-decay} and \ref{thm:3D-decay} and their proofs (this choice may be significant for Theorem \ref{thm:3D_weak_sol}).

For Leray-Hopf weak solutions $(u,B)$ of \eqref{HMHD}, we can express $(u,B)$ as an integral form
\begin{equation}\label{integral_form}
\begin{aligned}
u(t,x) &= e^{t\Delta}u_0(x)-\int^t_0 e^{(t-\tau)\Delta}\mathbb{P}{\nabla}\cdot\left(u\otimes u- B\otimes B\right)(\tau,x)\,d\tau, \\
B(t,x) &= e^{t\Delta}B_0(x)+\int^t_0 e^{(t-\tau)\Delta}{\nabla}\times\left(u\times B-\nabla\cdot(B\otimes B)\right)(\tau,x)\,d\tau,
\end{aligned}
\end{equation}
in $\mathcal{S}'$ and for almost every $0<t<\infty$.
By taking the Fourier transform to \eqref{integral_form}, we obtain
\begin{equation}\label{Fourier_integral_form}
\begin{aligned}
\widehat{u}(t,\xi) &=e^{-t|\xi|^2}\widehat{u}_0(\xi)-\int^t_0e^{-(t-\tau)|\xi|^2}m(\xi) i{\xi}\cdot\left(\widehat{u\otimes u}-\widehat{B\otimes B}\right)(\tau,\xi)\,d\tau, \\
\widehat{B}(t,\xi) &=e^{-t|\xi|^2}\widehat{B}_0(\xi)+\int^t_0e^{-(t-\tau)|\xi|^2} i{\xi}\times\left(\widehat{u\times B}-i\xi\cdot(\widehat{B\otimes B})\right)(\tau,\xi)\,d\tau
\end{aligned}
\end{equation}
where $m(\xi)$ is a matrix valued Fourier multiplier of the Leray projection operator $\mathbb{P}$ defined as $m_{ij}(\xi)=\delta_{ij}-\xi_i\xi_j/|\xi|^2$. As described in Section \ref{sec:intro}, when $d=2$, the system is understood in the $2\frac12$-dimensional sense with $\nabla$ and $\xi$ replaced by $\widetilde{\nabla}=(\partial_1,\partial_2,0)$ and $\widetilde{\xi}=(\xi_1,\xi_2,0)$, respectively. The justification of \eqref{integral_form} and \eqref{Fourier_integral_form} for Leray-Hopf weak solutions will be discussed in Appendix \ref{sec:app.C}.

Now, we recall the low frequency pseudomeasure space $\mathcal{Y}^\sigma_{l}$
\[
\mathcal{Y}^\sigma_l := \lt\{f \in \mathcal{S}'(\R^2) \ | \ \sup_{|\xi|\leq 1}  |\xi|^\sigma |\widehat{f}(\xi)| <\infty\rt\},
\]
and note that $\mathcal{Y}^{\sigma_1}_l\hookrightarrow\mathcal{Y}^{\sigma_2}_l$ when $\sigma_1\leq\sigma_2$. Below, we present three technical lemmas to be used throughout this section.

\begin{lemma}\label{lem:heat_kernel}
The heat kernel $e^{t\Delta}$ in $\R^d$ has the following estimates:
\begin{enumerate}
\item For $1\le p\le q\le \infty$,
\[
\|e^{t\Delta}f\|_{L^q}\le C(p,q) t^{-\frac{d}{2}\lt(\frac{1}{p}-\frac{1}{q}\rt)}\|f\|_{L^p},\quad \| e^{t\Delta}\nabla f\|_{L^q}\le C(p,q) t^{-\frac{d}{2}\lt(\frac{1}{p}-\frac{1}{q}\rt)-\frac{1}{2}}\|f\|_{L^p}
\]

\item For $\sigma<\frac{d}{2}$,
\[
\|e^{t\Delta}f\|_{L^2}^2\leq C t^{-\frac{d}{2}+\sigma}\lt(\|f\|_{\mathcal{Y}^\sigma_l}^2+\|f\|_{L^2}^2\rt).
\]
\end{enumerate}
\end{lemma}
\begin{proof}
Since the estimate (1) is well-known, we only show the second relation. We use the Plancherel theorem to obtain
\[\begin{aligned}
\|e^{t\Delta}f\|_{L^2}^2 &= \int_{\R^d} e^{-|\xi|^2 t} |\hat f(\xi)|^2\,d\xi\\
& = \lt(\int_{\{ |\xi|\le 1\} } + \int_{\{ |\xi|>1\}}\rt)e^{-|\xi|^2 t} |\hat f(\xi)|^2\,d\xi\\
&\le \|f\|_{\mathcal{Y}_\ell^\sigma}^2 \int_{\{|\xi|\le 1\}} |\xi|^{-2\sigma} e^{-|\xi|^2 t}\,d\xi + \int_{\{ |\xi|>1\}} |\xi|^{d-2\sigma} e^{-|\xi|^2 t}|\hat f(\xi)|^2\,d\xi\\
&\le C\|f\|_{\mathcal{Y}_\ell^\sigma}^2 \int_0^1 r^{-2\sigma +d-1}e^{-r^2 t}\,dr  + Ct^{-\frac d2+\sigma} \|f\|_{L^2}^2\\
&\le C\|f\|_{\mathcal{Y}_\ell^\sigma}^2 t^{-\frac{d}{2}+\sigma}\int_0^t \tau^{-\sigma +\frac d2 -1}e^{-\tau}\,d\tau  + Ct^{-\frac d2+\sigma} \|f\|_{L^2}^2\\
&\le Ct^{-\frac{d}{2}+\sigma}\lt(\|f\|_{\mathcal{Y}^\sigma_l}^2+\|f\|_{L^2}^2\rt),
\end{aligned}\]
where we also used, for $a, k>0$,

\[
x^k e^{-ax} \le \lt(\frac{k}{a} \rt)^k e^{-k}, \quad \mbox{on } \ x>0, 
\]
\end{proof}

\begin{lemma}\label{lem:time_convolution}
For $0<\alpha,\beta<1$, 
\[
\int^t_0(t-\tau)^{-\alpha}\tau^{-\beta}\,d\tau= t^{1-\alpha-\beta}\int^1_0(1-\theta)^{-\alpha}\theta^{-\beta}\,d\theta=t^{1-\alpha-\beta}\mathcal{B}(1-\alpha,1-\beta)
\]
where $\mathcal{B}(\cdot,\cdot)$ is the beta function which is bounded as long as its variables are positive.
\end{lemma}
\begin{proof}
The proof follows from the change of variables in integrations, and the definition of the beta function.
\end{proof}

\begin{lemma}\label{lem:Fourier_splitting}
Let $\epsilon\geq0$ and $f$ be a smooth function satisfying 
\[
\frac{d}{dt}\|f(t)\|_{L^2}^2+\|\nabla f(t)\|_{L^2}^2\leq \epsilon(1+t)^{-1-\alpha}
\]
for all $t>0$. Suppose there exists a positive constant $C_0>0$ and $\sigma<\frac{d}{2}$ such that 
\[
\sup_{t\in[0,\infty)}\|f(t)\|_{\mathcal{Y}^{\sigma}_l}\leq C_0.
\]
Then,
\[
\|f(t)\|_{L^2}^2\leq C(1+t)^{-\frac{d}{2}+\sigma}+C\epsilon(1+t)^{-\alpha}
\]
for all $t\geq t_0$ for some $t_0\geq 0$.
\end{lemma}
\begin{proof}
The proof is a slight modification of \cite[Lemma 2.1]{BJS} and hence we omit the details.
\end{proof}

We also provide the $L^2$ decay of derivatives to Leray-Hopf weak soltuions.
\begin{proposition}\label{prop:higher-decay}
Let $d=2, 3$, and let $(u,B)$ be a global Leray-Hopf weak solution of \eqref{HMHD} corresponding to $(u_0,B_0)\in L^2_\sigma(\R^d)$. Suppose that there exists $t_0\ge0$ such that $(u,B)$ is strong (and so smooth) on $[t_0,\infty)$. Then for each $k\in\N$
\[
\|\nabla^k u(t)\|_{L^2}^2 + \|\nabla^k B(t)\|_{L^2}^2\leq C_k(1+t-t_0)^{-k}\quad \text{for all $t\ge t_0$.}
\]
Additionally if $\|B(t)\|_{L^2}^2\le C(1+t)^{-\mu}$ for all $t\ge 0$ and $\mu>0$, then there exists $T_0\ge t_0$ such that
\[
\|\nabla^k B(t)\|_{L^2}^2\leq C_k(1+t-T_0)^{-k-\mu}\quad \text{for all $t\ge T_0$.}
\]
\end{proposition}

\begin{proof}
The first part of Proposition \ref{prop:higher-decay} follows by the standard argument for nonlinear heat equations, using the Sobolev interpolation theory. Hence, we only deal with the second part: 
\[
\|\nabla^k B(t)\|_{L^2}^2\leq C_k(1+t-T_0)^{-k-\mu}\quad \text{for all $t\ge T_0$,}
\]
provided that $\|B(t)\|_{L^2}^2\le C(1+t)^{-\mu}$.

Since the proof for $k\ge 2$ is analogous, we only consider $k=1$. By the energy estimates on $B$, we have
\[
\begin{aligned}
\frac{1}{2}\frac{d}{dt}\|\nabla B\|_{L^2}^2+\|\Delta B\|_{L^2}^2 &=\int_{\R^d} \Delta B\cdot(u\cdot{\nabla}B-B\cdot{\nabla}u)\,dx+\int_{\R^d}\Delta B\cdot{\nabla}\times(J\times B)\,dx \\
&\le C\left(\|\nabla B\|_{L^d}+\|B\|_{L^\infty}\right)\|\nabla u\|_{L^2}\|\Delta B\|_{L^2}+ C\|\nabla B\|_{L^d}\|\Delta B\|_{L^2}^2 \\
&\le C\|\nabla u\|_{L^2}\|B\|_{L^2}^{1-\frac{d}{4}}\|\Delta B\|_{L^2}^{1+\frac{d}{4}}+ C\|\nabla B\|_{L^d}\|\Delta B\|_{L^2}^2 \\
&\le C\|\nabla u\|_{L^2}^{\frac{8}{4-d}}\|B\|_{L^2}^2+\left(\frac{1}{4}+C_0\|\nabla B\|_{L^d}\right)\|\Delta B\|_{L^2}^2
\end{aligned}
\]
for all $t\ge t_0$. The first part of Proposition \ref{prop:higher-decay} implies that there exists $T_0\ge t_0$ such that
\[
\|\nabla B(t)\|_{L^d} \le \frac{1}{4C_0} \quad \mbox{for all } \ t \ge T_0.
\]
Hence, we deduce that
\[
\frac{d}{dt}\|\nabla B\|_{L^2}^2+\|\Delta B\|_{L^2}^2\le C\|\nabla u\|_{L^2}^{\frac{8}{4-d}}\|B\|_{L^2}^2 \le C(1+t-T_0)^{-\frac{4}{4-d}-\mu}\le C(1+t-T_0)^{-2-\mu}\quad \mbox{for all } \ t \ge T_0.
\]
Since 

\[
\|\nabla B\|_{L^2}^4\le C\|B\|_{L^2}^2\|\Delta B\|_{L^2}^2\le C(1+t-T_0)^{-\mu}\|\Delta B\|_{L^2}^2,
\]
we arrive at

\[
\frac{d}{dt}\|\nabla B\|_{L^2}^2+C_1(1+t-T_0)^{\mu}\|\nabla B\|_{L^2}^4\le C_2(1+t-T_0)^{-2-\mu}\quad \mbox{for all } \ t \ge T_0.
\]
By letting $f(t):=(1+t-T_0)^\mu \|\nabla B(t)\|_{L^2}^2$, we get

\[
\begin{aligned}
\frac{df}{dt}&=\mu(1+t-T_0)^{\mu-1}\|\nabla B\|_{L^2}^2+(1+t-T_0)^\mu \frac{d}{dt}\|\nabla B\|_{L^2}^2 \\
&\le \mu(1+t-T_0)^{\mu-1}\|\nabla B\|_{L^2}^2-C_1(1+t-T_0)^{2\mu}\|\nabla B\|_{L^2}^4+C_2(1+t-T_0)^{-2}\\
&\le -C_3 f^2+C_4(1+t-T_0)^{-2}
\end{aligned}
\]
where we used the Young's inequality in the last inequality to get proper constants $C_i>0$. Then we also write $F(t):=f(t)+2C_4(1+t-T_0)^{-1}$ and obtain
\[
\begin{aligned}
\frac{dF}{dt}=\frac{df}{dt}-2C_4(1+t-T_0)^{-2} \le-C_3 f^2-C_4(1+t-T_0)^{-2} \le -C_0 F^2,
\end{aligned}
\]
which implies $F(t)\le C(1+t-T_0)^{-1}$ for all $t\geq T_0$. This yields our desired decay rates 
\[
\|\nabla B(t)\|_{L^2}^2\le C(1+t-T_0)^{-1-\mu}\quad\text{for all } \ t\ge T_0.
\]
\end{proof}

\begin{remark}\upshape
We present some remarks on Proposition \ref{prop:higher-decay}.
\begin{enumerate}
\item For two dimensional Leray-Hopf weak solutions, this proposition is unconditional from Theorem \ref{thm:2D-eventual-regularity}. However, since there is no eventual smoothness of three dimensional Leray-Hopf weak solutions, the proposition is conditional, so one may consider global strong solutions of \eqref{HMHD} for $d=3$, whose global existence is guaranteed for sufficiently small data \cite{Chae Degond Liu, DT22, Tan}

\item In fact, if 
\[
\|u(t)\|_{L^2}^2\le C(1+t)^{-\mu_1},\quad \|B(t)\|_{L^2}^2\le C(1+t)^{-\mu_2}
\]
for all $t\ge 0$ and $\mu_1,\mu_2> 0$, then we can derive
\[
\|\nabla^k u(t)\|_{L^2}^2\leq C_k(1+t-T_0)^{-k-\mu_1},\quad \|\nabla^k B(t)\|_{L^2}^2\leq C_k(1+t-T_0)^{-k-\mu_2}\quad \text{for all $t\ge T_0$.}
\]
The proof is similar to that of Proposition \ref{prop:higher-decay}, but the decay rates of $\|\nabla^k u(t)\|_{L^2}$ may require finitely many iteration of the process. We omit the proof because Proposition \ref{prop:higher-decay} is enough for proving Theorem \ref{thm:2D-decay}.
\end{enumerate}
\end{remark}

\subsection{Proof of Theorem \ref{thm:2D-decay}}

We deduce from Proposition \ref{prop:non-uniform} and Proposition \ref{prop:higher-decay} that for any sufficiently small constant $\delta>0$, there exists $T_0>0$ such that for each $k\in\N$, $(u,B)\in C([T_0,\infty);H^k(\R^2))$ and 
\[
\|u(t)\|_{L^2}+\|B(t)\|_{L^2}\le \delta, \quad \|\nabla^k u(t)\|_{L^2}^2+\|\nabla^k B(t)\|_{L^2}^2\leq C_k(1+t-T_0)^{-k},
\]
for all $t\geq T_0$ and we will provide a specific choice for $\delta>0$ later. 

We also need to check that $\|u(T_0)\|_{\mathcal{Y}^{\sigma_1}_l}, \|B(T_0)\|_{\mathcal{Y}^{\sigma_2}_l}<\infty$ provided that $(u_0,B_0)\in \mathcal{Y}^{\sigma_1}_l(\R^2)\times\mathcal{Y}^{\sigma_2}_l(\R^2)$. 
From the integral form \eqref{Fourier_integral_form}, we have
\[
\begin{aligned}
|\xi|^{\sigma_1}|\widehat{u}(T_0,\xi)| &\le |\xi|^{\sigma_1} e^{-T_0|\xi|^2}|\widehat{u}_0(\xi)|+\int^{T_0}_0|\xi|^{1+\sigma_1}e^{-(T_0-\tau)|\xi|^2}\lt(\|u(\tau)\|_{L^2}^2+\|B(\tau)\|_{L^2}^2\rt)\,d\tau \\
&\le \|u_0\|_{\mathcal{Y}^{\sigma_1}_l}+C\mathcal{E}_0\int^{T_0}_0(T_0-\tau)^{-\frac{1+\sigma_1}{2}}\,d\tau\le \|u_0\|_{\mathcal{Y}^{\sigma_1}_l}+C\mathcal{E}_0 {T_0}^{\frac{1-\sigma_1}{2}},
\end{aligned}
\]
\[
\begin{aligned}
|\xi|^{\sigma_2}|\widehat{B}(T_0,\xi)| &\le |\xi|^{\sigma_2} e^{-T_0|\xi|^2}|\widehat{B}_0(\xi)|+\int^{T_0}_0|\xi|^{1+\sigma_2}e^{-(T_0-\tau)|\xi|^2}\|u(\tau)\|_{L^2}\|B(\tau)\|_{L^2}\,d\tau\\
&\quad+\int^{T_0}_0|\xi|^{2+\sigma_2}e^{-(T_0-\tau)|\xi|^2} \|B(\tau)\|_{L^2}^2\,d\tau \\
&\le \|B_0\|_{\mathcal{Y}^{\sigma_2}_l}+C\mathcal{E}_0\int^{T_0}_0(T_0-\tau)^{-\frac{1+\sigma_2}{2}}\,d\tau\le \|B_0\|_{\mathcal{Y}^{\sigma_2}_l}+C\mathcal{E}_0 {T_0}^{\frac{1-\sigma_2}{2}},
\end{aligned}
\]
for $|\xi|\le 1$, $\sigma_1,\sigma_2<1$ and $\mathcal{E}_0:=\|u_0\|_{L^2}^2+\|B_0\|_{L^2}^2$. Hence, we obtain
\[
\|u(T_0)\|_{\mathcal{Y}^{\sigma_1}_l}\le \|u_0\|_{\mathcal{Y}^{\sigma_1}_l}+C\mathcal{E}_0 T^{\frac{1-\sigma_1}{2}}_0<\infty,\quad \|B(T_0)\|_{\mathcal{Y}^{\sigma_2}_l}\le \|B_0\|_{\mathcal{Y}^{\sigma_2}_l}+C\mathcal{E}_0 T^{\frac{1-\sigma_2}{2}}_0<\infty.
\]

\noindent $\bullet$ (\textbf{Case 1}: $B_0\in\mathcal{Y}^{\sigma_2}_l$ for $\sigma_2\in(-1,1)$)  Since $(u,B)$ is smooth for all $t\ge T_0$, we consider the integral form of \eqref{2.5HMHD} with respect to $B$ starting from $t=T_0$:
\[
B(t,x)= e^{(t-T_0)\Delta}B(T_0,x)+\int^t_{T_0}e^{(t-\tau)\Delta}\widetilde{\nabla}\times\left(u\times B-B\cdot\widetilde{\nabla} B\right)(\tau,x)\,d\tau.
\]
By using Lemma \ref{lem:heat_kernel}, we obtain
\[\begin{aligned}
\|B(t)\|_{L^2} &\leq C (t-T_0)^{-\frac{1-\sigma_2}{2}}\lt(\|B(T_0)\|_{\mathcal{Y}^{\sigma_2}_l}+\|B(T_0)\|_{L^2}\rt)+\int^t_{T_0}\lt\|e^{(t-\tau)\Delta}\widetilde{\nabla}\times(u\times B-B\cdot\widetilde{\nabla} B)(\tau)\rt\|_{L^2}\,d\tau\\
&\le C(T_0) (t-T_0)^{-\frac{1-\sigma_2}{2}}+\int^t_{T_0} (t-\tau)^{-\frac{1}{p}}\lt(\|(u\times B)(\tau)\|_{L^p}+\|(B\cdot\nabla B)(\tau)\|_{L^p}\rt)\,d\tau \\
&\le C(T_0) (t-T_0)^{-\frac{1-\sigma_2}{2}}+\int^t_{T_0} (t-\tau)^{-\frac{1}{p}}\lt(\|u(\tau)\|_{L^{\frac{2p}{2-p}}}+\|\nabla B(\tau)\|_{L^{\frac{2p}{2-p}}}\rt)\|B(\tau)\|_{L^2}\,d\tau
\end{aligned}\]
for all $t >T_0$, where $1<p<\min{(\frac{2}{1-\sigma_2},2)}$. Since
\[\begin{aligned}
& \|u(t)\|_{L^{\frac{2p}{2-p}}}\leq \|u(t)\|_{L^2}^{\frac{2}{p}-1}\|\nabla u(t)\|_{L^2}^{2-\frac{2}{p}}\leq C\delta^{\frac{2}{p}-1}(1+t-T_0)^{-1+\frac{1}{p}},\\
& \|\nabla B(t)\|_{L^{\frac{2p}{2-p}}}\leq \|B(t)\|_{L^2}^{\frac{1}{p}-\frac{1}{2}}\|\nabla^2 B(t)\|_{L^2}^{\frac{3}{2}-\frac{1}{p}}\leq C\delta^{\frac{1}{p}-\frac{1}{2}}(1+t-T_0)^{-\frac{3}{2}+\frac{1}{p}},
\end{aligned}\]
we arrive at
\[\begin{aligned}
& (t-T_0)^{\frac{1-\sigma_2}{2}}\|B(t)\|_{L^2} \\
&\leq C(T_0)+C\delta^{\frac{1}{p}-\frac{1}{2}}(t-T_0)^{\frac{1-\sigma_2}{2}}\int^t_{T_0}(t-\tau)^{-\frac{1}{p}}(\tau-T_0)^{-1+\frac{1}{p}}\|B(\tau)\|_{L^2}\,d\tau \\
&\leq C(T_0)+C\delta^{\frac{1}{p}-\frac{1}{2}}\sup_{\tau\in[T_0,t]}\lt[(\tau-T_0)^{\frac{1-\sigma_2}{2}}\|B(\tau)\|_{L^2}\rt] (t-T_0)^{\frac{1-\sigma_2}{2}}\int^t_{T_0}(t-\tau)^{-\frac{1}{p}}(\tau-T_0)^{-1+\frac{1}{p}-\frac{1-\sigma_2}{2}}\,d\tau \\
&\leq C(T_0)+C_0\delta^{\frac{1}{p}-\frac{1}{2}}\sup_{\tau\in[T_0,t]}\lt[(\tau-T_0)^{\frac{1-\sigma_2}{2}}\|B(\tau)\|_{L^2}\rt],
\end{aligned}\]
where we used Lemma \ref{lem:time_convolution} in the last inequality.
Hence, by choosing $\delta^{\frac{1}{p}-\frac{1}{2}}<\frac{1}{2C_0}$ we derive
\[
\sup_{\tau\in[T_0,t]}\lt[(\tau-T_0)^{\frac{1-\sigma_2}{2}}\|B(\tau)\|_{L^2}\rt]\le C(T_0).
\]
Together with \eqref{energy ineq}, we get the desired decay rates of \textbf{Case 1}. \\

\noindent $\bullet$ (\textbf{Case 2}: $u_0\in\mathcal{Y}^{\sigma_1}_l$ and $B_0\in\mathcal{Y}^{\sigma_2}_l$ for $\sigma_1, \sigma_2\in[-1,1)\times [-1,1)$ except for $(\sigma_1,\sigma_2)=(-1,0)$) 
For this case, we split the proof into three subcases:\\

\noindent $\diamond$ (Case 2-1: $\sigma_1,\sigma_2\in (-1,1)\times (-1,1)$) Since we have the desired decay rates of $B$ in this case from \textbf{Case 1}, we only consider the decay rates of $u$. Similar to $B$ in \textbf{Case 1}, we consider the integral form of \eqref{HMHD} with respect to $u$ starting from $t=T_0$:
\[
u(t,x)=e^{(t-T_0)\Delta}u(T_0,x)-\int^t_{T_0}e^{(t-\tau)\Delta}\mathbb{P}\widetilde{\nabla}\cdot(u \otimes u-B\otimes B)(\tau,x)\,d\tau.
\]
From the integral form, we obtain
\begin{align*}
\|u(t)\|_{L^2} &\le \|e^{(t-T_0)\Delta}u(T_0,x)\|_{L^2}+\int^t_{T_0}\lt\|e^{(t-\tau)\Delta}\mathbb{P}\widetilde{\nabla}\cdot(u\otimes u- B\otimes B)(\tau)\rt\|_{L^2}\,d\tau \\
&\le \|e^{(t-T_0)\Delta}u(T_0,x)\|_{L^2}+\int^t_{T_0}(t-\tau)^{-\frac{1}{p}}\lt(\|u(\tau)\|_{L^{2p}}^2+\|B(\tau)\|_{L^{2p}}^2\rt)\,d\tau
\end{align*}
for any $1<p<2$. Since
\begin{align*}
& \|u(t)\|_{L^{2p}}^2\le \|u(t)\|_{L^2}^{\frac{2}{p}}\|\nabla u(t)\|_{L^2}^{2-\frac{2}{p}}\le C\delta^{\frac{2}{p}-1}(1+t-T_0)^{-1+\frac{1}{p}}\|u(t)\|_{L^2}, \\
& \|B(t)\|_{L^{2p}}^2\le \|B(t)\|_{L^2}^{\frac{2}{p}}\|\nabla B(t)\|_{L^2}^{2-\frac{2}{p}}\le C(1+t-T_0)^{-1+\frac{1}{p}-1+\sigma_2},
\end{align*}
we have
\begin{align*}
\|u(t)\|_{L^2} &\le \|e^{(t-T_0)\Delta}u(T_0,x)\|_{L^2}+C\delta^{\frac{2}{p}-1}\int^t_{T_0}(t-\tau)^{-\frac{1}{p}}(1+\tau-T_0)^{-1+\frac{1}{p}}\|u(\tau)\|_{L^2}\,d\tau \\
&\quad +C\int^t_{T_0}(t-\tau)^{-\frac{1}{p}}(1+\tau-T_0)^{-1+\frac{1}{p}-1+\sigma_2}\,d\tau.
\end{align*}
When $\sigma_1\ge -1+2\sigma_2$, using Lemma \ref{lem:heat_kernel} we obtain
\[
\|e^{(t-T_0)\Delta}u(T_0,x)\|_{L^2}\le C(t-T_0)^{-\frac{1-\sigma_1}{2}}\lt(\|u(T_0)\|_{\mathcal{Y}^{\sigma_1}_l}+\|u(T_0)\|_{L^2}\rt).
\]
Then, we choose $p<\frac{2}{1-\sigma_1}$ to get
\[
\int^t_{T_0}(t-\tau)^{-\frac{1}{p}}(1+\tau-T_0)^{-1+\frac{1}{p}-\frac{1}{2}+\frac{\sigma_1}{2}}\le (t-T_0)^{-\frac{1-\sigma_1}{2}},
\]
which implies
\begin{equation}\label{case2-1-1}
(t-T_0)^{\frac{1-\sigma_1}{2}}\|u(t)\|_{L^2}\le C(T_0)+C_0\delta^{\frac{2}{p}-1}\sup_{\tau\in[T_0,t]}\lt[(\tau-T_0)^{\frac{1-\sigma_1}{2}}\|u(\tau)\|_{L^2}\rt].
\end{equation}
Similarly when $\sigma_1<-1+2\sigma_2$, using Lemma \ref{lem:heat_kernel} we obtain
\[
\|e^{(t-T_0)\Delta}u(T_0)\|_{L^2}\le C(t-T_0)^{-1+\sigma_2}\left(\|u(T_0)\|_{\mathcal{Y}^{-1+2\sigma_2}_l}+\|u(T_0)\|_{L^2}\right)\le C(t-T_0)^{-1+\sigma_2}\left(\|u(T_0)\|_{\mathcal{Y}^{\sigma_1}_l}+\|u(T_0)\|_{L^2}\right).
\]
Then, we choose $p<\frac{1}{1-\sigma_2}$ for $0<\frac{1+\sigma_1}{2}<\sigma_2<1$ to get
\[
\int^t_{T_0}(t-\tau)^{-\frac{1}{p}}(\tau-T_0)^{-1+\frac{1}{p}-1+\sigma_2}\le C(t-T_0)^{-1+\sigma_2},
\]
which implies
\begin{equation}\label{case2-1-2}
(t-T_0)^{1-\sigma_2}\|u(t)\|_{L^2}\le C(T_0)+C_0\delta^{\frac{2}{p}-1}\sup_{\tau\in[T_0,t]}\lt[(\tau-T_0)^{1-\sigma_2}\|u(\tau)\|_{L^2}\rt].
\end{equation}
Hence, by choosing $\delta^{\frac{2}{p}-1}<\frac{1}{2C_0}$ in \eqref{case2-1-1} and \eqref{case2-1-2} respectively, we derive
\[
\sup_{\tau\in[T_0,t]}\lt[(\tau-T_0)^{\frac{1}{2}-\frac{1}{2}\max{(\sigma_1,-1+2\sigma_2)}}\|u(\tau)\|_{L^2}\rt]\le C(T_0),
\]
which implies the desired decay rates of $u$.\\

\noindent $\diamond$ (Case 2-2: $\sigma_1=-1$ and $\sigma_2\in (-1,0)\cup (0,1)$)
When $\sigma_1=-1$ and $\sigma_2>0$, then $u_0\in \mathcal{Y}^{-1}_l\hookrightarrow \mathcal{Y}^{-1+2\sigma_2}_l$ and we deduce the desired decay rates
\[
\|u(t)\|_{L^2}^2\le C(1+t)^{-1+(-1+2\sigma_2)}
\]
from Case 2-1. Thus, we only consider $\sigma_2\le 0$. 

When $\sigma_1=-1$ and $-1<\sigma_2\le 0$, then $u_0\in \mathcal{Y}^{-1}_l\hookrightarrow \mathcal{Y}^{-1+\epsilon}_l$ for $0<\epsilon<1+\sigma_2\le 1$. From Case 2-1, we have
\begin{equation}\label{case2-2:u_decay}
\|u(t)\|_{L^2}^2\le C(1+t)^{-2+\epsilon}.
\end{equation}
We also recall from \textbf{Case 1} and Proposition \ref{prop:higher-decay} that there exists $T_0>0$ such that for all $t\geq T_0$
\begin{equation}\label{case2-2:B_decay}
\|B(t)\|_{L^2}^2\le C(1+t)^{-1+\sigma_2},\quad 
\|\nabla B(t)\|_{L^2}^2\leq C(1+t-T_0)^{-2+\sigma_2}.
\end{equation}
This implies
\[
\frac{1}{2}\frac{d}{dt}\|u\|_{L^2}^2+\|\nabla u\|_{L^2}^2=\int_{\R^2} (B\cdot\nabla B)\cdot u\,dx\leq \|\nabla u\|_{L^2}\|B\|_{L^4}^2  \leq \frac{1}{2}\|\nabla u\|_{L^2}^2+C\|B\|_{L^2}^2\|\nabla B\|_{L^2}^2,
\]
and hence
\begin{equation}\label{case2-2:u_ineq}
\frac{d}{dt}\|u(t)\|_{L^2}^2+\|\nabla u(t)\|_{L^2}^2\leq C(1+t-T_0)^{-3+2\sigma_2}, \quad\text{for all $t>T_0$.}
\end{equation}
Moreover, for $|\xi|\leq1$ and $t\ge T_0$ we have
\[\begin{aligned}
|\xi|^{-1}|\widehat{u}(t,\xi)| &\le |\xi|^{-1} e^{-t|\xi|^2}|\widehat{u}(T_0,\xi)|+\int^t_{T_0} e^{-(t-\tau)|\xi|^2}\lt(\|u(\tau)\|_{L^2}^2+\|B(\tau)\|_{L^2}^2\rt)\,d\tau \\
&\le \|u(T_0)\|_{\mathcal{Y}^{-1}_l}+C\int^t_{T_0} (1+\tau)^{-1+\sigma_2}\,d\tau \le C
\end{aligned}\]
for $\sigma_2<0$. So we deduce from Lemma \ref{lem:Fourier_splitting} that for $\sigma_1=-1$ and $\sigma_2<0$
\[
\|u(t)\|_{L^2}^2\leq C(1+t-T_0)^{-2}+C(1+t-T_0)^{-2+2\sigma_2}\leq C(1+t-T_0)^{-2}.
\]

\vspace{0.1cm}

\noindent $\diamond$ (Case 2-3: $\sigma_1\in[-1,1)$ and $\sigma_2=-1$)
Since $B_0\in\mathcal{Y}^{-1}_l \hookrightarrow \mathcal{Y}^{-1+\epsilon}_l$ for sufficiently small $\epsilon>0$, the decay rates in Case 2-1 and Case 2-2 imply that
\[
\|u(t)\|_{L^2}^2\le C(1+t)^{-1+\sigma_1},\quad \|B(t)\|_{L^2}^2\le C(1+t)^{-2+\epsilon}\quad \text{for all $t\ge0$.}
\]
Thus, we only consider the decay rates of $B$.
We first define a new function
\[
M(t,x)= B(t,x)e^{-2\lambda\int^t_0 \|\nabla u(\tau)\|_{L^2}^2\,d\tau},\quad\text{for all $t\ge T_0$,}
\]
where $\lambda>0$ will be chosen later.
Since $2\int^t_0\|\nabla u(\tau)\|_{L^2}^2\,d\tau\le \mathcal{E}_0:=\|u_0\|_{L^2}^2+\|B_0\|_{L^2}^2$ from \eqref{energy ineq}, we have
\[
e^{-\lambda\mathcal{E}_0}|B(t,x)|\le |M(t,x)| \le |B(t,x)|,\quad e^{-\lambda\mathcal{E}_0}|\widehat{B}(t,\xi)|\le |\widehat{M}(t,\xi)| \le |\widehat{B}(t,\xi)|,
\]
for all $t\ge0$ and $x,\xi\in\R^2$. Then, we derive the equation of $M$:
\[
\pa_t M +2\lambda\|\nabla u\|_{L^2}^2 M-\Delta M+(u\cdot\nabla)M-(M\cdot\nabla) u+ e^{2\lambda\int^t_0 \|\nabla u(\tau)\|_{L^2}^2\,d\tau}\nabla\times((\nabla\times M)\times M)=0.
\]
Hence, we obtain
\[
\frac{1}{2}\frac{d}{dt}\|M\|_{L^2}^2+2\lambda\|\nabla u\|_{L^2}^2\|M\|_{L^2}^2+\|\nabla M\|_{L^2}^2 =\int (M\cdot\nabla) u\cdot M \le \frac{1}{2}\|\nabla M\|_{L^2}^2+C_0 \|\nabla u\|_{L^2}^2\|M\|_{L^2}^2,
\]
which implies 
\[
\frac{d}{dt}\|M\|_{L^2}^2+\|\nabla M\|_{L^2}^2\le 0, \quad \text{for all $t\ge T_0$}
\]
by setting $\lambda=C_0/2$.

Now, we investigate the $\mathcal{Y}^{-1}_l$-norm of $B$, which is comparable to that of $M$: for $|\xi|\le 1$ and $t>T_0$
\begin{align*}
|\xi|^{-1}|\widehat{B}(t,\xi)| &\le \|B(T_0)\|_{\mathcal{Y}^{-1}_l}+\int^t_{T_0} e^{-(t-\tau)|\xi|^2}\|u(\tau)\|_{L^2}\|B(\tau)\|_{L^2}\,d\tau+\int^t_{T_0}|\xi|e^{-(t-\tau)|\xi|^2}\|B(\tau)\|_{L^2}^2\,d\tau \\
&\le \|B(T_0)\|_{\mathcal{Y}^{-1}_l}+\int^t_{T_0} (1+\tau)^{-\frac{1-\sigma_1}{2}}(1+\tau)^{-1+\frac{\epsilon}{2}}\,d\tau+\int^t_{T_0}(t-\tau)^{-\frac{1}{2}}(1+\tau)^{-2+\epsilon}\,d\tau \le C
\end{align*}
if we choose $\epsilon<1-\sigma_1$. Therefore, by using Lemma \ref{lem:Fourier_splitting} with respect to $M$, we deduce that 
\[
\|M(t)\|_{L^2}^2\le C(1+t)^{-2},
\]
which implies the desired decay rates of $B$.\\

\noindent $\bullet$ (\textbf{Case 3}: $u_0\in\mathcal{Y}^{\sigma_1}_l$ and $B_0\in\mathcal{Y}^{\sigma_2}_l$ for $\sigma_1=-1$, $\sigma_2=0$) 
In this case, we borrow some inequalities obtained in Case 2-2, which hold when $\sigma_1=-1$, $\sigma_2=0$ as well. We use \eqref{case2-2:u_decay} and \eqref{case2-2:B_decay} to obtain that for $|\xi|\le 1$ and $t\ge T_0$
\[\begin{aligned}
|\xi|^{-1}|\widehat{u}(t,\xi)| &\le |\xi|^{-1} e^{-(t-T_0)|\xi|^2}|\widehat{u}(T_0,\xi)|+\int^t_{T_0} e^{-(t-\tau)|\xi|^2}\lt(\|u(\tau)\|_{L^2}^2+\|B(\tau)\|_{L^2}^2\rt)\,d\tau \\
&\le \|u(T_0)\|_{\mathcal{Y}^{-1}_l}+C\int^t_{T_0} (1+\tau)^{-1}\,d\tau \le C\log{(e+ t)}.
\end{aligned}\]
Slightly changing the proof of Lemma \ref{lem:Fourier_splitting} with \eqref{case2-2:u_ineq}, we derive
\[
\|u(t)\|_{L^2}^2\le C(1+t)^{-2}\lt(\log{(e+t)}\rt)^2.
\]

\subsection{Proof of Theorem \ref{thm:3D-decay}}
By Proposition \ref{prop:non-uniform} and Sobolev interpolation, a global strong solution of \eqref{HMHD} satisfies 
\[
\|u(t)\|_{L^3}\rightarrow 0\quad\text{as $t\rightarrow\infty$.}
\]
Hence, from the second equation of \eqref{HMHD} we easily get 
\begin{equation}\label{energy of B}
\frac{1}{2}\frac{d}{dt}\|B\|_{L^2}^2+\|\nabla B\|_{L^2}^2=\int_{\R^3}\nabla\times(u\times B)\cdot B\,dx\le C\|u\|_{L^3}\|\nabla B\|_{L^2}^2\le \frac{1}{2}\|\nabla B\|_{L^2}^2,\quad\text{for all $t\ge T_0$}
\end{equation}
for some $T_0>0$.\\

\noindent $\bullet$ (\textbf{Case 1}: $B_0\in\mathcal{Y}^{\sigma_2}_l$ for $\sigma_2\in[-1,\frac{3}{2})$) 
Similar to the proof of Theorem \ref{thm:2D-decay}, we consider the integral form of \eqref{HMHD} with respect to $B$ starting from $t=T_0$:
\[
B(t,x)=e^{(t-T_0)\Delta}B(T_0,x)+\int^t_{T_0} e^{(t-\tau)\Delta}\nabla\times(u\times B-B\cdot\nabla B)(\tau,x)\,d\tau.
\]
We deduce from this integral form that for all $|\xi|\le 1$ and $t\ge T_0$ 
\begin{equation}\label{Y_norm of B}
\begin{aligned}
&|\xi|^{\sigma_2}|\widehat{B}(t,\xi)| \\
&\quad\le \|B(T_0)\|_{\mathcal{Y}^{\sigma_2}_l}+\int^t_{T_0} |\xi|^{1+\sigma_2}e^{-(t-\tau)|\xi|^2}\|u(\tau)\|_{L^2}\|B(\tau)\|_{L^2}+\int^t_{T_0}|\xi|^{2+\sigma_2}e^{-(t-\tau)|\xi|^2}\|B(\tau)\|_{L^2}^2\,d\tau \\
&\quad\le \|B(T_0)\|_{\mathcal{Y}^{\sigma_2}_l}+C\sqrt{\mathcal{E}_0}\int^t_{T_0}|\xi|^{1+\sigma_2} e^{-(t-\tau)|\xi|^2}\|B(\tau)\|_{L^2}\,d\tau.
\end{aligned}
\end{equation}
For this case, we split the proof into three subcases:\\

\noindent $\diamond$ (Case 1-1: $\sigma_2\in [1,\frac{3}{2})$) For $|\xi|\le 1$ and $t\ge T_0$, we have
\[
|\xi|^{\sigma_2}|\widehat{B}(t,\xi)|\le \|B(T_0)\|_{\mathcal{Y}^{\sigma_2}_l}+C\mathcal{E}_0\int^t_{T_0}|\xi|^{2}e^{-(t-\tau)|\xi|^2}\,d\tau \le C.
\] 
Hence, we obtain the desired decay rates by Lemma \ref{lem:Fourier_splitting} with \eqref{energy of B}.\\

\noindent $\diamond$ (Case 1-2: $\sigma_2\in (-1,1)$)
For Case 1-2, we first suppose that $\|B(t)\|_{L^2}^2\le C(1+t)^{-1+\sigma_2}$. This assumption follows from Case 1-1 inductively. Then, from \eqref{Y_norm of B}, we have
\[
|\xi|^{\sigma_2}|\widehat{B}(t,\xi)|\le \|B(T_0)\|_{\mathcal{Y}^{\sigma_2}_l}+C\int^t_{T_0}(t-\tau)^{-\frac{1}{2}-\frac{\sigma_2}{2}}(1+\tau)^{-\frac{1}{2}+\frac{\sigma_2}{2}}\,d\tau \le C,
\]
which implies the desired decay rates by Lemma \ref{lem:Fourier_splitting}.\\

\noindent $\diamond$ (Case 1-3: $\sigma_2=-1$)
Since $\mathcal{Y}^{-1+\epsilon}_l\hookrightarrow\mathcal{Y}^{-1}_l$ for any $\epsilon>0$, $\|B(t)\|_{L^2}^2\le C(1+t)^{-\frac{5}{2}+\epsilon}$ provided that $B_0\in\mathcal{Y}^{-1}_l$ from Case 1-2. Again from \eqref{Y_norm of B}, we arrive at
\[
|\xi|^{-1}|\widehat{B}(t,\xi)|\le \|B(T_0)\|_{\mathcal{Y}^{-1}_l}+C\int^t_{T_0}(1+\tau)^{-\frac{5}{4}+\frac{\epsilon}{2}}\,d\tau \le C.
\]
Hence, we have (1) of Theorem \ref{thm:3D-decay}\\

\noindent $\bullet$ (\textbf{Case 2}: $u_0\in\mathcal{Y}^{\sigma_1}_l$ for $\sigma_1\in[1,\frac{3}{2})$) From the first equation of \eqref{HMHD}, we obtain
\[
\frac{1}{2}\frac{d}{dt}\|u\|_{L^2}^2+\|\nabla u\|_{L^2}^2=\int_{\R^3}(B\cdot\nabla B)\cdot u\,dx \le \|B\|_{L^3}\|\nabla B\|_{L^2}\|\nabla u\|_{L^2}\le \frac{1}{2}\|\nabla u\|_{L^2}^2+C\|B\|_{L^2}\|\nabla B\|_{L^2}^3.
\] 
Hence, from Proposition \ref{prop:higher-decay} we have
\begin{equation}\label{energy of u 1}
\frac{d}{dt}\|u\|_{L^2}^2+\|\nabla u\|_{L^2}^2\le C\|B\|_{L^2}\|\nabla B\|_{L^2}^3\le C(1+t)^{-\frac{3}{2}}.
\end{equation}
On the other hand, we consider the integral form of \eqref{HMHD} with respect to $u$ starting from $t=T_0$:

\[
u(t,x)=e^{(t-T_0)\Delta}u(T_0,x)-\int^t_{T_0} e^{(t-\tau)\Delta}\mathbb{P}\nabla\cdot(u\otimes u-B\otimes B)(\tau,x)\,d\tau.
\]
For all $|\xi|\le1$ and $t\ge T_0$
\[\begin{aligned}
|\xi|^{\sigma_1}|\widehat{u}(t,\xi)| &\le \|u(T_0)\|_{\mathcal{Y}^{\sigma_1}_l}+\int^t_{T_0}|\xi|^{1+\sigma_1} e^{-(t-\tau)|\xi|^2}\lt(\|u(\tau)\|_{L^2}^2+\|B(\tau)\|_{L^2}^2\rt)\,d\tau \\
&\le \|u(T_0)\|_{\mathcal{Y}^{\sigma_1}_l}+\mathcal{E}_0\int^t_{T_0}|\xi|^{2} e^{-(t-\tau)|\xi|^2}\,d\tau\le C.
\end{aligned}\]
We deduce from Lemma \ref{lem:Fourier_splitting} with \eqref{energy of u 1} the desired decay rates in (2) of Theorem \ref{thm:3D-decay}.\\

\noindent $\bullet$ (\textbf{Case 3}: $(u_0,B_0)\in\mathcal{Y}^{\sigma_1}_l\times \mathcal{Y}^{\sigma_2}_l$ for $\sigma_1\in[-1,1)$, $\sigma_2\in[-1,\frac{3}{2})$) From \textbf{Case 1} and Proposition \ref{prop:higher-decay}, we have
\begin{equation}\label{energy of u 2}
\frac{d}{dt}\|u\|_{L^2}^2+\|\nabla u\|_{L^2}^2\le C\|B\|_{L^2}\|\nabla B\|_{L^2}^3 \le C(1+t)^{-\frac{9}{2}+2\sigma_2}.
\end{equation}
For this case, we split the proof into two subcases:\\

\noindent $\diamond$ (Case 3-1: $\sigma_1\in[0,1)$)
From \textbf{Case 1} and \textbf{Case 2}, we have
\[
\|u(t)\|_{L^2}^2 + \|B(t)\|_{L^2}^2\le C(1+t)^{-\frac{1}{2}} + C(1+t)^{-\frac{3}{2}+\sigma_2}\le C(1+t)^{-\frac{1}{2}+\frac{1}{2}\max{\{0,-2+2\sigma_2\}}}.
\]
For the sake of notational convenience, we define $\mu:=\max{\{\sigma_1,-2+2\sigma_2\}}$. Then,
\[\begin{aligned}
|\xi|^{\mu}|\widehat{u}(t,\xi)| &\le |\xi|^{\mu}|\widehat{u}(T_0,\xi)|+\int^{t}_{T_0}|\xi|^{1+\mu} e^{-(t-\tau)|\xi|^2}\lt(\|u(\tau)\|_{L^2}^2+\|B(\tau)\|_{L^2}^2\rt)\,d\tau \\
&\le |\xi|^{\sigma_1}|\widehat{u}(T_0,\xi)|+\int^{t}_{T_0}(t-\tau)^{-\frac{1}{2}-\frac{\mu}{2}}(1+\tau)^{-\frac{1}{2}+\frac{1}{2}\max{\{0,-2+2\sigma_2\}}}\,d\tau \\
&\le \|u(T_0)\|_{\mathcal{Y}^{\sigma_1}_l}+\int^{t}_{T_0}(t-\tau)^{-\frac{1}{2}-\frac{\mu}{2}} (1+\tau)^{-\frac{1}{2}+\frac{\mu}{2}}\,d\tau \le C.
\end{aligned}\]
Hence, by Lemma \ref{lem:Fourier_splitting} with \eqref{energy of u 2} we obtain the desired decay. \\

\noindent $\diamond$ (Case 3-2: $\sigma_1\in[-1,0)$)
From \textbf{Case 1} and Case 3-1, we have
\[
\|u(t)\|_{L^2}^2 + \|B(t)\|_{L^2}^2\le C(1+t)^{-\frac{3}{2}+\max{\{0,-2+2\sigma_2\}}}+C(1+t)^{-\frac{3}{2}+\sigma_2}\le C(1+t)^{-\frac{3}{2}+\max{\{0,\sigma_2\}}}\le C(1+t)^{-\frac{1}{2}+\frac{\mu}{2}},
\]
where we used $\max{\{-2,-2+2\sigma_2\}}\le \max{\{\sigma_1,-2+2\sigma_2\}}=\mu$. Then, for $\mu\neq-1$
\[\begin{aligned}
|\xi|^{\mu}|\widehat{u}(t,\xi)| &\le |\xi|^{\mu}|\widehat{u}(T_0,\xi)|+\int^{t}_{T_0}|\xi|^{1+\mu} e^{-(t-\tau)|\xi|^2}\lt(\|u(\tau)\|_{L^2}^2+\|B(\tau)\|_{L^2}^2\rt)\,d\tau \\
&\le \|u(T_0)\|_{\mathcal{Y}^{\sigma_1}_l}+\int^{t}_{T_0}(t-\tau)^{-\frac{1}{2}-\frac{\mu}{2}} (1+\tau)^{-\frac{1}{2}+\frac{\mu}{2}}\,d\tau \le C,
\end{aligned}\]
which implies the desired decay rates.
When $\mu=-1$, so $\sigma_1=-1$ and $\sigma_2\le \frac{1}{2}$, 
\[\begin{aligned}
|\xi|^{-1}|\widehat{u}(t,\xi)| &\le |\xi|^{-1}|\widehat{u}(T_0,\xi)|+\int^{t}_{T_0}e^{-(t-\tau)|\xi|^2}\lt(\|u(\tau)\|_{L^2}^2+\|B(\tau)\|_{L^2}^2\rt)\,d\tau \\
&\le \|u(T_0)\|_{\mathcal{Y}^{-1}_l}+\int^{t}_{T_0}(1+\tau)^{-\frac{3}{2}+\max{\{0,\sigma_2\}}}\,d\tau.
\end{aligned}\]
This is uniformly bounded when $\sigma_1=-1$ and $\sigma_2<\frac{1}{2}$, and bounded by $\log{(e+t)}$ when $\sigma_1=-1$ and $\sigma_2=\frac{1}{2}$. This completes the proof from Lemma \ref{lem:Fourier_splitting} (or its modification) with \eqref{energy of u 2}.

\section*{Acknowledgments}

The work of J. Jung was supported by NRF grant no. 2022R1A2C1002820.
The work of J. Shin was supported by the National Research Foundation of Korea(NRF) funded by the Ministry of Education(RS-2025-25437873), and the Ministry of Science and ICT(RS-2024-00406821).

\appendix

\section{Integral formulation for Leray-Hopf weak solutions}\label{sec:app.C}

The passage from the weak formulation to the integral formulation (or Duhamel formulation) is standard. Since we use this formulation for Leray-Hopf weak solutions of \eqref{HMHD} whose nonlinear terms are only controlled in the energy class, we include the details for completeness.

Let $(u,B)$ be a global Leray-Hopf weak solution of \eqref{HMHD} for $d=2, 3$. Then, $(u,B)$ solves \eqref{HMHD u} and \eqref{HMHD B} in the sense of distributions. More precisely, for any divergence-free vector field $\Phi\in C^\infty_c([0,\infty)\times \R^d)$
\begin{equation}\label{u dist eq}
-\int^\infty_0\langle u,\pa_t\Phi+\nu\Delta\Phi\rangle\,dt=\langle u_0,\Phi(0)\rangle+\int^\infty_0\int_{\R^d}\lt(u\otimes u-B\otimes B\rt):\nabla\Phi\,dxdt,
\end{equation}
and for any vector field $\Psi\in C^\infty_c([0,\infty)\times \R^d)$
\begin{equation}\label{B dist eq}
-\int^\infty_0\langle B,\pa_t\Psi+\eta\Delta\Psi\rangle\,dt=\langle B_0,\Psi(0)\rangle+\int^\infty_0\langle u\times B,\nabla\times\Psi\rangle\,dt+\int^\infty_0\int_{\R^d} B\otimes B :\nabla(\nabla\times\Psi)\,dxdt
\end{equation}
where $\langle f, g\rangle =\int_{\R^d} f\cdot g\,dx$. For notational simplicity, we write $\nabla$ and $\xi$ throughout this appendix, and for $d=2$, $\nabla$ and $\xi$ should be replaced by $\widetilde\nabla=(\partial_1,\partial_2,0)$ and $\widetilde\xi=(\xi_1,\xi_2,0)$, respectively. 

Although the weak formulation is initially stated for compactly supported test functions in space, \eqref{u dist eq} and \eqref{B dist eq} hold for test functions in $C^\infty_c([0,\infty);\mathcal{S}(\R^d))$. Indeed the energy bound gives $u(t),B(t)\in L^2(\R^d)\subset\mathcal{S}'(\R^d)$ and $(u\otimes u)(t), (B\otimes B)(t), (u\times B)(t) \in L^1(\R^d)\subset\mathcal{S}'(\R^d)$ for almost every $t$.

Now, we choose a proper test function to derive the integral form in the sense of distributions. We first consider \eqref{HMHD u}, the equation of $u$. For a fixed time $t$ (among almost everywhere in $(0,\infty)$), and any divergence-free $\phi\in\mathcal{S}(\R^d)$, we define
\[
\phi_t(\tau)=e^{\nu(t-\tau)\Delta}\phi,\quad 0\le\tau\le t.
\]
Then, for any $\tau\in[0,t]$ $\phi_t$ satisfies
\[
\pa_\tau\phi_t(\tau)+\nu\Delta\phi_t(\tau)=0,\quad\phi_t(t)=\phi,
\]
and $\nabla\cdot\phi_t(\tau)=0$, $\phi_t(\tau)\in\mathcal{S}(\R^d)$.
Let $\chi_\epsilon\in C_c^\infty([0,\infty))$ satisfy
\[
\chi_\epsilon(\tau)=1 \quad\text{for }0\le \tau\le t-\epsilon,
\qquad
\chi_\epsilon(\tau)=0 \quad\text{for }\tau\ge t,
\]
and $0\le \chi_\epsilon\le1$. We set
\[
\Phi_\epsilon(\tau,x):=\chi_\epsilon(\tau)\phi_t(\tau,x).
\]
Then $\Phi_\epsilon\in C_c^\infty([0,\infty);\mathcal S(\mathbb R^d))$
and $\Phi_\epsilon(0,x)=e^{\nu t\Delta}\phi(x)$.
Then we easily get
\[
\pa_\tau\Phi_\epsilon+\nu\Delta\Phi_\epsilon=\chi'_\epsilon(\tau)\phi_t(\tau)+\chi_\epsilon(\tau)\lt(\pa_\tau\phi_t+\nu\Delta\phi_t\rt)=\chi'_\epsilon(\tau)\phi_t(\tau,x).
\]
By substituting $\Phi=\Phi_\epsilon$ into \eqref{u dist eq}, we obtain
\[
-\int^t_0 \chi'_\epsilon(\tau)\langle u(\tau),\phi_t(\tau)\rangle\,d\tau=\langle u_0,\phi_t(0)\rangle+\int^t_0\chi_\epsilon(\tau)\int_{\R^d}\lt(u\otimes u-B\otimes B\rt)(\tau):\nabla\phi_t(\tau)\,dxd\tau.
\]

Here $-\chi_\epsilon'$ is an approximate identity concentrated at the terminal time $t$ from the left. Hence, if $t$ is a Lebesgue point of the scalar function $ \tau\mapsto \langle u(\tau),\phi\rangle$, then 
\[
-\int^t_0 \chi'_\epsilon(\tau)\langle u(\tau),\phi_t(\tau)\rangle\,d\tau\rightarrow \langle u(t), \phi\rangle, \quad\text{as $\epsilon\rightarrow0$.}
\]
Indeed, the dependence of $\phi_t(\tau)\) on $\tau$ is harmless, since $e^{\nu(t-\tau)\Delta}\phi\to\phi$ strongly in $L^2$ as $\tau\uparrow t$. 
The convergence of the nonlinear term follows from dominated convergence, using 
\[
\|u(\tau)\otimes u(\tau)\|_{L^1}+\|B(\tau)\otimes B(\tau)\|_{L^1} \le \mathcal E_0,\quad \sup_{0\le \tau\le t}
\|\nabla e^{\nu(t-\tau)\Delta}\phi\|_{L^\infty}<\infty.
\]
Hence, we deduce the dual form of the integral formulation for almost every $t$
\[
\langle u(t),\phi\rangle =\langle u_0, e^{\nu t\Delta}\phi\rangle+\int^t_0\int_{\R^d}\lt((u\otimes u-B\otimes B)\rt)(\tau):\nabla e^{\nu(t-\tau)\Delta}\phi\,dxd\tau.
\]  
Therefore, we arrive at the integral formulation
\begin{equation}\label{Duhamel u}
u(t)=e^{\nu t\Delta}u_0-\int^t_0 e^{\nu(t-\tau)\Delta}\mathbb{P}\nabla\cdot\lt(u\otimes u-B\otimes B\rt)(\tau)\,d\tau
\end{equation}
in $\mathcal{S}'(\R^d)$ for almost every $t\in (0,\infty)$. 
Similar to $u$, we derive the integral formulation of $B$ from \eqref{HMHD B}: 
\begin{equation}\label{Duhamel B}
B(t)= e^{\eta t\Delta}B_0+\int^t_0 e^{\eta(t-\tau)\Delta}\nabla\times(u\times B-\nabla\cdot(B\otimes B))(\tau)\,d\tau,
\end{equation}
in $\mathcal{S}'(\R^d)$ for almost every $t\in(0,\infty)$. 

Moreover, since the integral formulation is written in $\mathcal{S}'(\R^d)$, we obtain the integral formulation in Fourier form:
\begin{equation}\label{Fourier Duhamel}
\begin{aligned}
\widehat{u}(t,\xi) &=e^{-\nu t|\xi|^2}\widehat{u}_0(\xi)-\int^t_0e^{-\nu(t-\tau)|\xi|^2}m(\xi) i{\xi}\cdot\left(\widehat{u\otimes u}-\widehat{B\otimes B}\right)(\tau,\xi)\,d\tau, \\
\widehat{B}(t,\xi) &=e^{-\eta t|\xi|^2}\widehat{B}_0(\xi)+\int^t_0e^{-\eta(t-\tau)|\xi|^2} i{\xi}\times\left(\widehat{u\times B}-i\xi\cdot(\widehat{B\otimes B})\right)(\tau,\xi)\,d\tau
\end{aligned}
\end{equation}
where $m(\xi)$ is a matrix valued Fourier multiplier of the Leray projection operator $\mathbb{P}$ defined as $m_{ij}(\xi)=\delta_{ij}-\xi_i\xi_j/|\xi|^2$. We note that the nonlinear terms in \eqref{Fourier Duhamel} are well-defined pointwise in $\xi$ for almost every $t$, since $u\otimes u, B\otimes B, u\times B\in L^1(\R^d)$ for almost every $t$. We also note that if $(u,B)$ is smooth near $t=t_0$, then the integral formulation \eqref{Duhamel u}, \eqref{Duhamel B} and its Fourier form \eqref{Fourier Duhamel} hold at $t=t_0$ since $\langle u(t),\phi\rangle$ and $\langle B(t),\phi\rangle$ for any $\phi\in\mathcal{S}(\R^d)$ have $t_0$ as a Lebesgue point for sure.

\section{Proof of Proposition \ref{prop:non-uniform}}\label{sec:app.A}

In this section, we prove Proposition \ref{prop:non-uniform}, the asymptotic $L^2$ decay of Leray-Hopf weak solutions of \eqref{HMHD}. All the arguments in the proof of Proposition \ref{prop:non-uniform} for $d=2$ and $d=3$ are exactly same except for the Sobolev embeddings, so we deal with both cases together. For simplicity, we let $\nu=\eta=1$ throughout this section, and present a lemma whose proof can be deduced from \cite{AS07, DL19}:

\begin{lemma}\label{lem:non-uniform}
Let $d=2$ or $d=3$. For a function $E(t) \in \mc^1(\R; \R_+)$ with $E(t) \ge0 $ and $\varphi(\xi) = e^{-|\xi|^2}$, $\psi(\xi) := 1- \varphi(\xi)$. Then a weak solution $(u,B)$ to \eqref{HMHD} satisfies

\[\begin{aligned}
\|\check \varphi\star u(t)\|_{L^2}^2 &\le \|e^{(t-s)\Delta}\check \varphi \star u(s)\|_{L^2}^2 - 2\int_s^t \langle u \cdot \nabla u(\tau), e^{2(t-\tau)\Delta}\check \varphi \star \check \varphi \star u(\tau)\rangle \,d\tau\\
&\quad + 2 \int_s^t \langle B \cdot \nabla B(\tau), e^{2(t-\tau)\Delta } \check \varphi \star \check \varphi \star u(\tau)\rangle \,d\tau,\\
E(t)\|\psi\hat u(t)\|_{L^2}^2 &\le E(s)\|\psi\hat u(s)\|_{L^2}^2 + \int_s^t E'(\tau) \|\psi \hat u(\tau)\|_{L^2}^2\,d\tau - 2 \int_s^t E(\tau) \| |\xi|\psi \hat u(\tau)\|_{L^2}^2 \,d\tau \\
&\quad -2\int_s^t E(\tau) \langle \widehat {u \cdot \nabla u} (\tau), \psi^2 \hat u(\tau)\rangle\,d\tau + 2\int_s^t E(\tau) \langle \widehat{B \cdot \nabla B}(\tau), \psi^2 \hat u(\tau)\rangle\,d\tau.
\end{aligned}\]
and

\[\begin{aligned}
\|\check \varphi\star B(t)\|_{L^2}^2 &\le \|e^{(t-s)\Delta}\check\varphi\star B(s)\|_{L^2}^2 - 2\int_s^t \langle u \cdot \nabla B(\tau), e^{2(t-\tau)\Delta}\check\varphi\star\check\varphi \star B(\tau)\rangle \,d\tau\\
&\quad + 2\int_s^t \langle B\cdot \nabla u(\tau), e^{2(t-\tau)\Delta}\check\varphi\star\check\varphi\star B(\tau)\rangle \,d\tau\\
&\quad -2\int_s^t \langle \nabla \times (\nabla \cdot (B\otimes B)) (\tau), e^{2(t-\tau)\Delta}\check\varphi\star\check\varphi\star B(\tau)\rangle\,d\tau,\\
E(t)\|\psi \hat B(t)\|_{L^2}^2 &\le E(s)\|\psi \hat B(s)\|_{L^2}^2 + \int_s^t E'(\tau) \|\psi \hat B(\tau)\|_{L^2}^2 \,d\tau - 2\int_s^t E(\tau) \| |\xi|\psi \hat B(\tau)\|_{L^2}^2\,d\tau\\
&\quad - 2\int_s^t E(\tau) \langle \widehat{u \cdot \nabla B}(\tau), \psi^2 \hat B(\tau)\rangle\,d\tau + 2\int_s^t E(\tau) \langle \widehat{B \cdot \nabla u}(\tau), \psi^2 \hat B(\tau)\rangle\,d\tau\\
&\quad - 2\int_s^t E(\tau) \langle \wh{\nabla \times (\nabla \cdot (B\otimes B))}(\tau), \psi^2 \hat B(\tau)\rangle \,d\tau.
\end{aligned}\]
\end{lemma}

\subsection{Proof of Proposition \ref{prop:non-uniform}}
Now, we show that for any weak solution $(u,B)$ to \eqref{HMHD} satisfies
\[
\lim_{t \to \infty}\lt(\|u(t)\|_{L^2}^2 + \|B(t)\|_{L^2}^2\rt) = 0.
\]
Using Plancherel theorem, we separately estimate low and high frequency parts as follows:

\[\begin{aligned}
\|u(t)\|_{L^2} &= \|\hat u(t)\|_{L^2}  \le \|\varphi \hat u\|_{L^2} + \|\psi \hat u\|_{L^2}\\
\|B(t)\|_{L^2} &= \|\hat B(t)\|_{L^2} \le \|\varphi \hat B\|_{L^2} + \|\psi \hat B\|_{L^2}
\end{aligned}\]

\noindent $\bullet$ (Part 1: Low frequency part) We first deduce from the previous lemma and Plancherel theorem that

\[\begin{aligned}
\|\varphi \hat u(t)\|_{L^2}^2 &\le \lt\|e^{-|\xi|^2 (t-s)}\varphi \hat u(s)\rt\|_{L^2}^2 + 2  \int_s^t \lt| \langle \widehat {u \cdot \nabla u}(\tau), e^{-2|\xi|^2(t-\tau)}\varphi^2  \hat u(\tau)\rangle \rt| \,d\tau\\
&\quad + 2 \int_s^t \lt|\langle \widehat{B \cdot \nabla B}(\tau), e^{-2|\xi|^2(t-\tau)}  \varphi^2 \hat u(\tau)\rangle \rt| \,d\tau,\\
&=: I_1 + I_2 +I_3.
\end{aligned}\]
First, we can use Lebesgue's dominated convergence theorem to have

\[
\lim_{t \to \infty} I_1 = 0.
\]
For $I_2$, we use Hausdorff-Young's inequality and Sobolev embedding to yield 
\[\begin{aligned}
I_2 &\le C\int_s^t \|\widehat{u \otimes u}\|_{L^d}\|\varphi^2\|_{L^{\frac{2d}{d-2}}}\||\xi| e^{-2|\xi|^2(t-\tau)}\hat u(\tau)\|_{L^2}\,d\tau\\
&\le C\int_s^t \|(u\otimes u)(\tau)\|_{L^{\frac{d}{d-1}}} \||\xi| \hat u(\tau)\|_{L^2} \,d\tau\\
&\le C\int_s^t \|u(\tau)\|_{L^{\frac{2d}{d-1}}}^2\|\nabla u(\tau)\|_{L^2}\,d\tau \le C\int_s^t \|u(\tau)\|_{L^2}\|\nabla u(\tau)\|_{L^2}^2\,d\tau \le C\int_s^t \|\nabla u(\tau)\|_{L^2}^2\,d\tau.
\end{aligned}\]
Similarly, we also have
\[\begin{aligned}
I_3 \le C\int_s^t \lt(\|\nabla B(\tau)\|_{L^2}^2 + \|\nabla u(\tau)\|_{L^2}^2 \rt)\,d\tau.
\end{aligned}\]
Thus,
\[
\lim_{t \to \infty}(I_2 + I_3) \le C\lim_{s \to \infty}\lim_{t \to \infty} \int_s^t  \lt(\|\nabla B(\tau)\|_{L^2}^2 + \|\nabla u(\tau)\|_{L^2}^2 \rt)\,d\tau = 0.
\]

For the magnetic field, we have
\[\begin{aligned}
\|\varphi \hat B(t)\|_{L^2}^2 &\le \lt\| e^{-|\xi|^2 (t-s)} \varphi \hat B(s)\rt\|_{L^2}^2 + 2 \int_s^t \lt|\langle \widehat{u \cdot \nabla B}(\tau), e^{-2|\xi|^2(t-\tau)}\varphi^2 \hat B(\tau)\rangle \rt|\,d\tau\\
&\quad+ 2 \int_s^t \lt|\langle \widehat{B\cdot \nabla u}(\tau), e^{-2|\xi|^2(t-\tau)}\varphi^2 \hat B(\tau)\rangle \rt|,d\tau\\
&\quad + 2\int_s^t \lt|\langle \wh{\nabla \times (\nabla \cdot (B\otimes B))}(\tau), e^{-2|\xi|^2(t-\tau)}\varphi^2 \hat B(\tau)\rangle \rt|\,d\tau\\
&=: J_1 + J_2 + J_3 + J_4.
\end{aligned}\]
By similar argument, one has
\[
\lim_{t\to\infty} J_1 = 0.
\]
For $J_2$,
\[\begin{aligned}
J_2 &\le C \int_s^t \lt|\langle \xi \widehat{u \otimes B}, e^{-2|\xi|^2(t-\tau)}\varphi^2 \hat{B}(\tau) \rt|\,d\tau\\
&\le C\int_s^t \|\widehat{u \otimes B}\|_{L^d} \|\varphi^2\|_{L^{\frac{2d}{d-2}}}\||\xi|e^{-2|\xi|^2(t-\tau)}\hat B(\tau)\|_{L^2}\,d\tau\\
&\le C\int_s^t \|(u \otimes B)(\tau)\|_{L^{\frac{d}{d-1}}} \||\xi|\hat B(\tau)\|_{L^2}\,d\tau\\
&\le C\int_s^t \|u(\tau)\|_{L^{\frac{2d}{d-1}}}\|B(\tau)\|_{L^{\frac{2d}{d-1}}}\|\nabla B\|_{L^2}\,d\tau \le C\int_s^t \lt(\|\nabla u(\tau)\|_{L^2}^2+\|\nabla B(\tau)\|_{L^2}^2\rt)\,d\tau.
\end{aligned}\]
Similarly, we also obtain
\[\begin{aligned}
J_3 \le C\int_s^t \lt(\|\nabla u(\tau)\|_{L^2}^2+\|\nabla B(\tau)\|_{L^2}^2\rt)\,d\tau,
\end{aligned}\]
and hence,
\[
\lim_{t \to \infty}(J_2 + J_3) \le \lim_{s \to \infty}\lim_{t \to \infty}C\int_s^t \lt(\|\nabla u(\tau)\|_{L^2}^2+\|\nabla B(\tau)\|_{L^2}^2\rt)\,d\tau =0.
\]
For $J_4$, one gets
\[\begin{aligned}
J_4 &\le C\int_s^t \lt|\langle \xi \times (\xi \cdot (\widehat{B\otimes B}(\tau))), e^{-2|\xi|^2(t-\tau)}\varphi^2 \hat B(\tau)  \rangle \rt|\,d\tau\\
&\le C \int_s^t \|\widehat{B\otimes B}(\tau)\|_{L^d}\||\xi|\varphi^2\|_{L^{\frac{2d}{d-2}}}\||\xi|e^{-2|\xi|^2(t-\tau)}\hat B(\tau)\|_{L^2}\,d\tau\\
&\le C\int_s^t \|B(\tau)\|_{L^2}\|\nabla B(\tau)\|_{L^2}^2\,d\tau \le C\int_s^t \|\nabla B(\tau)\|_{L^2}^2\,d\tau,
\end{aligned}\]
and again we have
\[
\lim_{t\to\infty} J_4 =0.
\]
Thus, we have the convergence of the low frequency part;
\[
\lim_{t\to \infty}\lt(\|\varphi \hat u(t)\|_{L^2}^2 + \|\varphi \hat B(t)\|_{L^2}^2\rt) =0.
\]

\noindent $\bullet$ (Part 2: High frequency part) Here, we follow the argument in \cite{AS07}. We choose $E(t) = (1+t)^\alpha$ and let $G^2(t) := \frac{\alpha}{2(1+t)}$ and $\chi(t) := \{ \xi \in \R^3 \ | \ |\xi|\le G(t)\}$ with $\alpha>3$. Then the previous lemma gives

\begin{align*}
E(t)& \lt[\|\psi\hat u(t)\|_{L^2}^2 + \|\psi \hat B(t)\|_{L^2}^2 \rt]\\
&\le E(s) \lt[\|\psi\hat u(s)\|_{L^2}^2 + \|\psi \hat B(s)\|_{L^2}^2 \rt]\\
&\quad + \int_s^t E'(\tau) \int_{\chi(\tau)} |\psi \hat u(\tau)|^2\,d\xi d\tau + \int_s^t E'(\tau) \int_{\chi(\tau)} |\psi \hat B(\tau)|^2\,d\xi d\tau\\
&\quad + \int_s^t E'(\tau) \int_{\R^3\setminus \chi(\tau)} |\psi \hat u(\tau)|^2\,d\xi d\tau + \int_s^t E'(\tau) \int_{\R^3\setminus \chi(\tau)} |\psi \hat B(\tau)|^2\,d\xi d\tau\\
&\quad - 2\int_s^t E(\tau)\| |\xi| \psi \hat u(\tau)\|_{L^2}^2\,d\tau - 2\int_s^t E(\tau) \| |\xi|\psi \hat B(\tau)\|_{L^2}^2 \,d\tau\\
&\quad -2\int_s^t E(\tau) \langle \widehat {u \cdot \nabla u} (\tau), \psi^2 \hat u(\tau)\rangle \,d\tau + 2\int_s^t E(\tau) \langle \widehat{B \cdot \nabla B}(\tau), \psi^2 \hat u(\tau)\rangle \,d\tau\\
&\quad -2\int_s^t E(\tau) \langle \widehat{u \cdot \nabla B}(\tau), \psi^2 \hat B(\tau)\rangle\,d\tau+ 2\int_s^t E(\tau) \langle \widehat{B \cdot \nabla u}(\tau), \psi^2 \hat B(\tau)\rangle\,d\tau\\
&\quad - 2\int_s^t E(\tau) \langle \wh{\nabla \times (\nabla \cdot (B\otimes B))}(\tau), \psi^2 \hat B(\tau)\rangle \,d\tau.
\end{align*}
Here, we may use $E'(t)-2E(t)G^2(t)=0$ to yield

\[
\begin{aligned}
\int_s^t& E'(\tau) \int_{\R^3\setminus \chi(\tau)} |\psi \hat u(\tau)|^2\,d\xi d\tau  - 2\int_s^t E(\tau)\| |\xi| \psi \hat u(\tau)\|_{L^2}^2\,d\tau\\
&+ \int_s^t E'(\tau) \int_{\R^3\setminus \chi(\tau)} |\psi \hat B(\tau)|^2\,d\xi d\tau - 2\int_s^t E(\tau) \| |\xi|\psi \hat B(\tau)\|_{L^2}^2 \,d\tau \le 0.
\end{aligned}
\]
Thus, we arrive at

\begin{align*}
E(t)& \lt[\|\psi\hat u(t)\|_{L^2}^2 + \|\psi \hat B(t)\|_{L^2}^2 \rt]\\
&\le E(s) \lt[\|\psi\hat u(s)\|_{L^2}^2 + \|\psi \hat B(s)\|_{L^2}^2 \rt]\\
&\quad + \int_s^t E'(\tau) \int_{\chi(\tau)} |\psi \hat u(\tau)|^2\,d\xi d\tau + \int_s^t E'(\tau) \int_{\chi(\tau)} |\psi \hat B(\tau)|^2\,d\xi d\tau\\
&\quad -2\int_s^t E(\tau) \langle \widehat {u \cdot \nabla u} (\tau), \psi^2 \hat u(\tau)\rangle \,d\tau + 2\int_s^t E(\tau) \langle \widehat{B \cdot \nabla B}(\tau), \psi^2 \hat u(\tau)\rangle \,d\tau\\
&\quad -2\int_s^t E(\tau) \langle \widehat{u \cdot \nabla B}(\tau), \psi^2 \hat B(\tau)\rangle\,d\tau+ 2\int_s^t E(\tau) \langle \widehat{B \cdot \nabla u}(\tau), \psi^2 \hat B(\tau)\rangle\,d\tau\\
&\quad - 2\int_s^t E(\tau) \langle \wh{\nabla \times (\nabla \cdot (B\otimes B))}(\tau), \psi^2 \hat B(\tau)\rangle \,d\tau\\
&=: \sum_{i=1}^8 K_i.
\end{align*}
For $K_2$ and $K_3$, we use $|\psi| = |1-\varphi| \le |\xi|^2 $ on $|\xi|\le 1$ to get

\[\begin{aligned}
K_2 + K_3 &\le \int_s^t E'(\tau) G^4(\tau) (\|u(\tau)\|_{L^2}^2 + \|B(\tau)\|_{L^2}^2)\,d\tau\\
&\le C(\|u_0\|_{L^2}^2 + \|B_0\|_{L^2}^2)\int_s^t (1+\tau)^{\alpha-3}\,d\tau \le  C\lt((1+t)^{\alpha-2} - (1+s)^{\alpha-2}\rt).
\end{aligned}\]
For $K_4$-$K_8$, we use $\langle u \cdot \nabla u, u\rangle = \langle u \cdot \nabla B, B\rangle=\langle \nabla\times (\nabla \cdot (B\otimes B)), B\rangle=0$ and $\langle B\cdot \nabla B, u\rangle + \langle B \cdot \nabla u, B\rangle =0$ to obtain

\[
\begin{aligned}
\sum_{i=4}^8 K_i &= -2\int_s^t E(\tau) \langle \widehat {u \cdot \nabla u} (\tau), (\psi^2-1) \hat u(\tau)\rangle \,d\tau + 2\int_s^t E(\tau) \langle \widehat{B \cdot \nabla B}(\tau), (\psi^2-1) \hat u(\tau)\rangle \,d\tau\\
&\quad -2\int_s^t E(\tau) \langle \widehat{u \cdot \nabla B}(\tau), (\psi^2-1) \hat B(\tau)\rangle\,d\tau+ 2\int_s^t E(\tau) \langle \widehat{B \cdot \nabla u}(\tau), (\psi^2-1) \hat B(\tau)\rangle\,d\tau\\
&\quad - 2\int_s^t E(\tau) \langle \wh{\nabla \times (\nabla \cdot (B\otimes B))}(\tau), (\psi^2-1) \hat B(\tau)\rangle \,d\tau\\
&=: \sum_{i=4}^8 \tilde{K}_i
\end{aligned}
\]
For $\tilde{K}_4$, with the fact that $\psi^2 -1 = e^{-2|\xi|^2} - 2e^{-|\xi|^2}$ belongs to the Schwartz class $\mathcal{S}(\R^3)$ in mind, one has

\[\begin{aligned}
\tilde{K}_4 &\le C\int_s^t E(\tau) \|\wh{u\otimes u}(\tau)\|_{L^d} \|\psi^2 -1\|_{L^{\frac{2d}{d-2}}}\||\xi|\hat u(\tau)\|_{L^2}\,d\tau\\
&\le C\int_s^t E(\tau) \|u(\tau)\|_{L^2}\|\nabla u(\tau)\|_{L^2}^2\,d\tau \le C\int_s^t E(\tau) \|\nabla u(\tau)\|_{L^2}^2\,d\tau.
\end{aligned}\]
Similar arguments imply

\[\begin{aligned}
\tilde{K}_5 + \tilde{K}_6 + \tilde{K}_7 \le C\int_s^t E(\tau) \lt(\|\nabla u(\tau)\|_{L^2}^2 + \|\nabla B(\tau)\|_{L^2}^2\rt) \,d\tau.
\end{aligned}\]
For $\tilde{K}_8$, 

\[\begin{aligned}
\tilde{K}_8 &\le C\int_s^t E(\tau) \|\wh{B\otimes B}(\tau)\|_{L^d}\| |\xi||\psi^2 - 1|\|_{L^{\frac{2d}{d-2}}}\||\xi|\hat B(\tau)\|_{L^2}\,d\tau\\
&\le C\int_s^t E(\tau) \|\nabla B(\tau)\|_{L^2}^2\,d\tau.
\end{aligned}\]
Hence, we gather all the estimates to yield

\[\begin{aligned}
 \lt[\|\psi\hat u(t)\|_{L^2}^2 + \|\psi \hat B(t)\|_{L^2}^2 \rt] &\le \frac{E(s)}{E(t)} \lt[\|\psi\hat u(s)\|_{L^2}^2 + \|\psi \hat B(s)\|_{L^2}^2 \rt]+\frac{C}{E(t)}\lt((1+t)^{\alpha-2} - (1+s)^{\alpha-2}\rt)\\
&\quad + C\int_s^t \frac{E(\tau)}{E(t)} \lt(\|\nabla u(\tau)\|_{L^2}^2 + \|\nabla B(\tau)\|_{L^2}^2 \rt)\,d\tau\\
&\le \lt(\frac{1+s}{1+t} \rt)^\alpha \lt[\|\psi\hat u(s)\|_{L^2}^2 + \|\psi \hat B(s)\|_{L^2}^2 \rt]+C(1+t)^{-2}\\
&\quad + C\int_s^t \lt(\|\nabla u(\tau)\|_{L^2}^2 + \|\nabla B(\tau)\|_{L^2}^2 \rt)\,d\tau
\end{aligned}\]
and this implies

\[\begin{aligned}
\lim_{t \to \infty}\lt[\|\psi\hat u(t)\|_{L^2}^2 + \|\psi \hat B(t)\|_{L^2}^2 \rt] &\le \lim_{t \to \infty}\lt(\frac{1+s}{1+t} \rt)^\alpha \lt[\|\psi\hat u(s)\|_{L^2}^2 + \|\psi \hat B(s)\|_{L^2}^2 \rt]\\
&\quad + \lim_{s\to\infty}\lim_{t\to\infty} C\int_s^t \lt(\|\nabla u(\tau)\|_{L^2}^2 + \|\nabla B(\tau)\|_{L^2}^2 \rt)\,d\tau\\
&\le 0
\end{aligned}\]
and therefore, we have attained

\[
\lim_{t \to \infty}\lt(\|u(t)\|_{L^2}^2 + \|B(t)\|_{L^2}^2 \rt)=0.
\]

\section{Proof of Proposition \ref{prop:non-uniform_B+omega}}\label{sec:app.B}

The proof is similar to that of Proposition \ref{prop:non-uniform}. First, we set $\Gamma=B+\omega$ and $\nu=1$, $\eta=2$ for simplicity. Then, in three dimension
\[
\pa_t\Gamma+u\cdot\nabla\Gamma-\Gamma\cdot\nabla u-\Delta\Gamma=\Delta B.
\]
Similar to Lemma \ref{lem:non-uniform}, we have the following lemma.
\begin{lemma}\label{lem:B+omega}
For a function $E(t) \in \mc^1(\R; \R_+)$ with $E(t) \ge0 $ and $\varphi(\xi) = e^{-|\xi|^2}$, $\psi(\xi) := 1- \varphi(\xi)$. Then a weak solution $(u,B)$ to \eqref{HMHD} with \eqref{3D_B+omega_bound} satisfies
\[\begin{aligned}
\|\check \varphi\star \Gamma(t)\|_{L^2}^2 &\le \|e^{(t-s)\Delta}\check\varphi\star \Gamma(s)\|_{L^2}^2 - 2\int_s^t \langle u \cdot \nabla \Gamma(\tau), e^{2(t-\tau)\Delta}\check\varphi\star\check\varphi \star \Gamma(\tau)\rangle \,d\tau\\
&\quad + 2\int_s^t \langle \Gamma\cdot \nabla u(\tau), e^{2(t-\tau)\Delta}\check\varphi\star\check\varphi\star \Gamma(\tau)\rangle \,d\tau +2\int_s^t \langle \Delta B (\tau), e^{2(t-\tau)\Delta}\check\varphi\star\check\varphi\star \Gamma(\tau)\rangle\,d\tau,\\
E(t)\|\psi \hat \Gamma(t)\|_{L^2}^2 &\le E(s)\|\psi \hat \Gamma(s)\|_{L^2}^2 + \int_s^t E'(\tau) \|\psi \hat \Gamma(\tau)\|_{L^2}^2 \,d\tau - 2\int_s^t E(\tau) \| |\xi|\psi \hat \Gamma(\tau)\|_{L^2}^2\,d\tau\\
&\quad - 2\int_s^t E(\tau) \langle \widehat{u \cdot \nabla \Gamma}(\tau), \psi^2 \hat \Gamma(\tau)\rangle\,d\tau + 2\int_s^t E(\tau) \langle \widehat{\Gamma \cdot \nabla u}(\tau), \psi^2 \hat \Gamma(\tau)\rangle\,d\tau \\
&\quad + 2\int_s^t E(\tau) \langle \wh{\Delta B}(\tau), \psi^2 \hat \Gamma(\tau)\rangle \,d\tau.
\end{aligned}\]
\end{lemma}
Using Plancherel theorem, we separately estimate low and high frequency parts as follows
\[
\|\Gamma(t)\|_{L^2}\le \|\varphi\hat \Gamma\|_{L^2}+\|\psi\hat\Gamma\|_{L^2}.
\]

We first estimate the low frequency part:
\[\begin{aligned}
\|\varphi \hat \Gamma(t)\|_{L^2}^2 &\le \lt\| e^{-|\xi|^2 (t-s)} \varphi \hat \Gamma(s)\rt\|_{L^2}^2 + 2 \int_s^t \lt|\langle \widehat{u \cdot \nabla \Gamma}(\tau), e^{-2|\xi|^2(t-\tau)}\varphi^2 \hat \Gamma(\tau)\rangle \rt|\,d\tau\\
&\quad+ 2 \int_s^t \lt|\langle \widehat{\Gamma\cdot \nabla u}(\tau), e^{-2|\xi|^2(t-\tau)}\varphi^2 \hat \Gamma(\tau)\rangle \rt|,d\tau  + 2\int_s^t \lt|\langle |\xi|^2\hat B(\tau), e^{-2|\xi|^2(t-\tau)}\varphi^2 \hat \Gamma(\tau)\rangle \rt|\,d\tau\\
&=: \mathcal{J}_1 + \mathcal{J}_2 + \mathcal{J}_3 + \mathcal{J}_4.
\end{aligned}\]
By the similar argument in the proof of Proposition \ref{prop:non-uniform}, we have
\[
\lim_{t\to\infty} \mathcal{J}_1 = 0,\quad
\lim_{t \to \infty}(\mathcal{J}_2 + \mathcal{J}_3) \le \lim_{s \to \infty}\lim_{t \to \infty}\int_s^t \|\nabla \Gamma(\tau)\|_{L^2}^2\,d\tau =0.
\]
For $\mathcal{J}_4$, one gets
\[\begin{aligned}
\mathcal{J}_4 &\le C\int_s^t \lt|\langle |\xi|^2 \hat B(\tau), e^{-2|\xi|^2(t-\tau)}\varphi^2 \hat \Gamma(\tau)  \rangle \rt|\,d\tau\\
&\le C \int_s^t \||\xi|\hat B(\tau)\|_{L^2}\|\varphi^2\|_{L^\infty}\||\xi|e^{-2|\xi|^2(t-\tau)}\hat \Gamma(\tau)\|_{L^2}\,d\tau\\
&\le C\int_s^t \|\nabla B(\tau)\|_{L^2}\|\nabla \Gamma(\tau)\|_{L^2}\,d\tau \le C\int_s^t \lt(\|\nabla B(\tau)\|_{L^2}^2+\|\nabla\Gamma(\tau)\|_{L^2}^2\rt)\,d\tau,
\end{aligned}\]
and again we have
\[
\lim_{t\to\infty} \mathcal{J}_4 =0.
\]
Thus, we have the convergence of the low frequency part;
\[
\lim_{t\to \infty} \|\varphi \hat \Gamma(t)\|_{L^2}^2 =0.
\]

Now we compute the hight frequency part. We choose $E(t) = (1+t)^\alpha$ and let $G^2(t) := \frac{\alpha}{2(1+t)}$ and $\chi(t) := \{ \xi \in \R^3 \ | \ |\xi|\le G(t)\}$ with $\alpha>3$. Then the previous lemma gives
\begin{align*}
&E(t) \|\psi \hat \Gamma(t)\|_{L^2}^2 \\
&\le E(s) \|\psi \hat \Gamma(s)\|_{L^2}^2 + \int_s^t E'(\tau) \int_{\chi(\tau)} |\psi \hat \Gamma(\tau)|^2\,d\xi d\tau + \int_s^t E'(\tau) \int_{\R^3\setminus \chi(\tau)} |\psi \hat \Gamma(\tau)|^2\,d\xi d\tau - 2\int_s^t E(\tau) \| |\xi|\psi \hat \Gamma(\tau)\|_{L^2}^2 \,d\tau\\
&\quad -2\int_s^t E(\tau) \langle \widehat{u \cdot \nabla \Gamma}(\tau), \psi^2 \hat \Gamma(\tau)\rangle\,d\tau+ 2\int_s^t E(\tau) \langle \widehat{\Gamma \cdot \nabla u}(\tau), \psi^2 \hat \Gamma(\tau)\rangle\,d\tau + 2\int_s^t E(\tau) \langle \wh{\Delta B}(\tau), \psi^2 \hat \Gamma(\tau)\rangle \,d\tau.
\end{align*}
Here, we may use $E'(t)-2E(t)G^2(t)=0$ to yield
\[
\begin{aligned}
\int_s^t E'(\tau) \int_{\R^3\setminus \chi(\tau)} |\psi \hat \Gamma(\tau)|^2\,d\xi d\tau - 2\int_s^t E(\tau) \| |\xi|\psi \hat \Gamma(\tau)\|_{L^2}^2 \,d\tau \le 0.
\end{aligned}
\]
Thus, we arrive at
\begin{align*}
E(t) \|\psi \hat \Gamma(t)\|_{L^2}^2 &\le E(s) \|\psi \hat \Gamma(s)\|_{L^2}^2 + \int_s^t E'(\tau) \int_{\chi(\tau)} |\psi \hat \Gamma(\tau)|^2\,d\xi d\tau \\
&\quad -2\int_s^t E(\tau) \langle \widehat{u \cdot \nabla \Gamma}(\tau), \psi^2 \hat \Gamma(\tau)\rangle\,d\tau+ 2\int_s^t E(\tau) \langle \widehat{\Gamma \cdot \nabla u}(\tau), \psi^2 \hat \Gamma(\tau)\rangle\,d\tau \\
&\quad + 2\int_s^t E(\tau) \langle \wh{\Delta B}(\tau), \psi^2 \hat \Gamma(\tau)\rangle \,d\tau \\
&=:\sum^5_{i=1}\mathcal{K}_i.
\end{align*}
For $\mathcal{K}_2$, we use $|\psi| = |1-\varphi| \le |\xi|^2 $ on $|\xi|\le 1$ to get
\[\begin{aligned}
\mathcal{K}_2 &\le \int_s^t E'(\tau) G^4(\tau) \|\Gamma(\tau)\|_{L^2}^2\,d\tau \le C\|\Gamma_0\|_{L^2}^2 \int_s^t (1+\tau)^{\alpha-3}\,d\tau \le  C\lt((1+t)^{\alpha-2} - (1+s)^{\alpha-2}\rt).
\end{aligned}\]
For $\mathcal{K}_3$ and $\mathcal{K}_4$, we obtain
\begin{align*}
\mathcal{K}_3+\mathcal{K}_4 &= -2\int^t_s E(\tau)\langle \widehat{u\otimes \Gamma}(\tau)-\widehat{\Gamma\otimes u}(\tau),\psi^2 \widehat{\nabla \Gamma}(\tau)\rangle\,d\tau \\
&\le C\int^t_s E(\tau)\|u(\tau)\|_{L^3}\|\Gamma(\tau)\|_{L^6}\|\psi^2\|_{L^\infty}\|\nabla\Gamma(\tau)\|_{L^2}\,d\tau \le C\int^t_s E(\tau)\|\nabla\Gamma(\tau)\|_{L^2}^2\,d\tau,
\end{align*}
since
\[
\|u\|_{L^3}\le C\lt(\|u\|_{L^2}+\|\omega\|_{L^2}\rt)\le C\lt(\|u\|_{L^2}+\|B\|_{L^2}+\|B+\omega\|_{L^2}\rt)\le C.
\]
For $\mathcal{K}_5$, we get
\begin{align*}
\mathcal{K}_5 \le C\int^t_s E(\tau)\|\psi^2\|_{L^\infty}\|\nabla B(\tau)\|_{L^2}\|\nabla \Gamma(\tau)\|_{L^2}\,d\tau \le C\int^t_s E(\tau)\lt(\|\nabla B(\tau)\|_{L^2}^2++\|\nabla \Gamma(\tau)\|_{L^2}^2\rt)\,d\tau 
\end{align*}

Hence, we gather all the estimates to yield
\[\begin{aligned}
\|\psi \hat \Gamma(t)\|_{L^2}^2 &\le \frac{E(s)}{E(t)} \|\psi \hat \Gamma(s)\|_{L^2}^2 +\frac{C}{E(t)}\lt((1+t)^{\alpha-2} - (1+s)^{\alpha-2}\rt)\\
&\quad + C\int_s^t \frac{E(\tau)}{E(t)} \lt(\|\nabla B(\tau)\|_{L^2}^2 + \|\nabla \Gamma(\tau)\|_{L^2}^2 \rt)\,d\tau\\
&\le \lt(\frac{1+s}{1+t} \rt)^\alpha \|\psi \hat \Gamma(s)\|_{L^2}^2 +C(1+t)^{-2} + C\int_s^t \lt(\|\nabla B(\tau)\|_{L^2}^2 + \|\nabla \Gamma(\tau)\|_{L^2}^2 \rt)\,d\tau
\end{aligned}\]
and this implies
\[\begin{aligned}
\lim_{t \to \infty} \|\psi \hat \Gamma(t)\|_{L^2}^2 &\le \lim_{t \to \infty}\lt(\frac{1+s}{1+t} \rt)^\alpha  \|\psi \hat \Gamma(s)\|_{L^2}^2  + \lim_{s\to\infty}\lim_{t\to\infty} C\int_s^t \lt(\|\nabla B(\tau)\|_{L^2}^2 + \|\nabla \Gamma(\tau)\|_{L^2}^2 \rt)\,d\tau \le 0
\end{aligned}\]
and therefore, we have attained

\[
\lim_{t \to \infty}\|\Gamma(t)\|_{L^2}^2 =0.
\]

\end{document}